\documentclass[12pt, leqno]{amsart}

\usepackage{fancyhdr}
\usepackage{layout}
\usepackage{float}
\usepackage{a4wide}
\usepackage{amssymb,amsmath,amsthm,amscd,amsxtra,amsfonts,amstext}

\usepackage{enumerate} 
\usepackage{enumitem} % to use enumerate style, e.g., \begin{enumerate}[label=(\roman*)], \arabic, \alph, \Alph
\usepackage{mathtools}
\usepackage[pdftex]{color,graphicx}
\usepackage{color}
\usepackage[dvipsnames]{xcolor}
\usepackage[vcentermath,enableskew,noautoscale]{youngtab}
\usepackage[poly,all]{xy}
\usepackage{verbatim}
\usepackage{setspace} % to use \begin{spacing}{1.5}...\end{spacing}
\usepackage{todonotes} % comment box
\usepackage[makeroom]{cancel}
\usepackage[normalem]{ulem} %\sout

\makeatletter
\@namedef{subjclassname@2020}{%
  \textup{2020} Mathematics Subject Classification}
\makeatother

\usepackage{yfonts} %to use font style \textfrak, \textswab
\usepackage{anyfontsize} % font size
\usepackage{t1enc} % font size
\usepackage{mathrsfs} % Use \mathscr
\usepackage{bm} %for blod math font
\usepackage{bbold} % for double-struck letters

\usepackage{amsrefs} %to use bibdiv
\usepackage{caption} % to remove a caption of a figure
\usepackage{subcaption} %for subfigure
\usepackage{lipsum}
\usepackage{here}

\usepackage{tikz}
\usetikzlibrary{matrix} %matrix in tikzpicture
\usetikzlibrary{arrows}
\usetikzlibrary{snakes} % brace
\usepackage{algorithmic}
\usepackage{latexsym}
\usepackage{epsfig}
\usepackage{url}
\usepackage{tikz} 
\usetikzlibrary{matrix}
\usetikzlibrary{calc}

\usepackage{fullpage, setspace}

\usepackage[CJKbookmarks,bookmarks=true,colorlinks=true]{hyperref} %dvipdfm,
\usepackage{hyperref}
\usepackage[right,displaymath,mathlines]{lineno}
\hypersetup{
    colorlinks=true,
    linkcolor=blue,
    filecolor=magenta,      
    urlcolor=magenta,
    citecolor=blue,
}
\newcommand{\veps}{\varepsilon}

\numberwithin{equation}{section}
\newtheorem{thm}{Theorem}[section]

\newtheorem{prop}[thm]{Proposition}
\newtheorem{lem}[thm]{Lemma}
\newtheorem{coro}[thm]{Corollary}
\newtheorem{rmk}[thm]{Remark}

\newtheorem{mainthm}{Theorem}

\begin{document}

\title{The Brjuno functions of the by-excess, odd, even and odd-odd continued fractions and their regularity properties}

\author{Seul Bee Lee}
\address{Center for Geometry and Physics, Institute for Basic Science (IBS), Pohang 37673, Korea}
\email{seulbee.lee@ibs.re.kr}

\author{Stefano Marmi}
\address{Scuola Normale Superiore, Palazzo della Carovana, Piazza dei Cavalieri 7, 56126 Pisa, Italy}
\email{stefano.marmi@sns.it}

%37F50: 	Small divisors, rotation domains and linearization in holomorphic dynamics
%37E05   Dynamical systems involving maps of the interval
%11J70  	Continued fractions and generalizations
%11K50  	Metric theory of continued fractions

\thanks{The first author acknowledges the support of the Centro di Ricerca Matematica Ennio de
Giorgi and the support of the Institute for Basic Science (IBS-R003-D1). All authors acknowledge the support of UniCredit Bank R\&D group for financial support through
the ‘Dynamics and Information Theory Institute’ at the Scuola Normale Superiore. The second author acknowledges the support through the PRIN Grant ``Regular and stochastic behaviour in
dynamical systems” (PRIN 2017S35EHN)}

\maketitle

\begin{abstract}
The Brjuno function was introduced by Yoccoz to study the linearizability of holomorphic germs and other one-dimensional small divisor problems.
The Brjuno functions associated with various continued fractions including the by-excess continued fraction were subsequently investigated: it was conjectured that the difference between the classical Brjuno function and the 
even part of the Brjuno function associated with the by-excess continued fraction extends to a H\"older continuous function of the whole real line.
In this paper, we prove this conjecture and we extend its validity to the more general case of Brjuno functions with positive exponents.

Moreover, we study the Brjuno functions associated to the odd and even continued fractions introduced by Schweiger. We show that they belong to all $L^p$ spaces, $p\ge1$. We prove that the Brjuno function associated to the odd continued fraction differs from the classical Brjuno function by a H\"older continuous function. On the other hand, the Brjuno function associated to the even continued fraction differs from the classical Brjuno function by a sum of a H\"older continuous function and a Brjuno-type function associated to the odd-odd continued fraction, introduced in the study of the best approximations of the form odd/odd.
\end{abstract}

\section{Introduction}\label{Sec:Intro}
Let $x\in\mathbb{R}\setminus\mathbb{Q}$: the \emph{regular continued fraction (RCF) expansion} of $x$ is defined by
\begin{equation}\label{eq:RCF}
x = a_0+\dfrac{1}{a_1+\dfrac{1}{\ddots+\dfrac{}{a_n+\ddots}}} =: [a_0;a_1,a_2,\cdots,a_n,\cdots],
\end{equation}
where $a_0\in \mathbb{Z}$ and $a_n\in\mathbb{N}$ for $n\ge 1$.
The \emph{principal convergents} of $x$ are the finite truncations $P_n(x)/Q_n(x) := [a_0;a_1,\cdots, a_n]$ for $n\ge 0$.
We call $x\in \mathbb{R}$ a \emph{Brjuno number} if $x$ satisfies the \emph{Brjuno condition}:
$\sum_{n=0}^{\infty}\frac{\log (Q_{n+1}(x))}{Q_n(x)} < \infty.$
The Brjuno condition was introduced by A.D. Brjuno \cites{Brj71,Brj72}:
in the case of germs of holomorphic diffeomorphisms
of one complex variable with an indifferent fixed point, extending the fundamental work of C.L. Siegel \cite{Sie42}, 
Brjuno proved that all germs with linear part $\lambda = e^{2\pi x i}$
are linearizable if $x$ is a Brjuno number. 
In 1988, J.-C. Yoccoz \cites{Yoc88, Yoc95} proved that 
this condition is also necessary. 
Similar results hold for the local conjugacy problem of analytic diffeomorphisms of the circle \cite{Yoc02},  for some complex area–preserving maps \cite{Mar90, Dav94, CM22}.
In a somewhat different but closely related context, it is conjectured that the Brjuno conditon is optimal for the existence of real analytic invariant circles in the standard family \cite{Mac88, Mac89, MS92}. See also \cite{BG01, Gen15} for related results.

Yoccoz's work used a rigorous renormalization technique and the $\mathrm{PGL}_2(\mathbb{Z})$ action on $\mathbb{R}\setminus\mathbb{Q}$ plays an important role: he thus found convenient to 
reformulate the Brjuno condition in terms of an arithemtical function, the \emph{Brjuno function}, which satisfies a cocycle
equation \cite{MMY01} under this action, see \cite[Appendix 5]{MMY01}.
The \emph{Gauss map} $G:(0,1]\to [0,1]$ is defined by $G(x) = \{1/x\}$, where $\{t\}$ is the fractional part of $t$. 
Let $x_n = G^n(\{x\})$, $\beta_{n}(x) = x_0x_{1}\cdots x_n$, $n\ge 0$, and $\beta_{-1}\equiv1$.
J.-C. Yoccoz \cite{Yoc88} defined\footnote{In his work Yoccoz actually uses the nearest-integer continued fraction map, which will be denoted $A_{1/2}$ in our work and is introduced below, instead of the Gauss map to define the Brjuno function. However the difference between the two functions is a positive H\"older continuous function which extends to the 
whole real line, as proved in \cite{MMY97}, thus 
they can both be used equivalently for investigating small divisor problems.}
the \emph{Brjuno function} $B:\mathbb{R}\setminus\mathbb{Q} \to \mathbb{R}\cup\{\infty\}$ defined by
\begin{equation*}\label{eq:def Brjuno}
B(x) = - \sum_{n=0}^\infty \beta_{n-1}(x)\log{x_n},
\end{equation*}
see Figure~\ref{fig:B} for the graph.
\begin{figure}
\includegraphics[width=7cm]{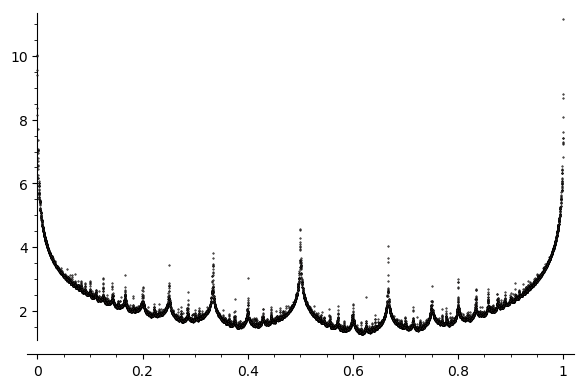} 
\caption{Numerical computation 
of Yoccoz's Brjuno function associated to the Gauss map at $10^4$ uniformly distributed random points in the interval $(0,1)$.}
\label{fig:B}
\end{figure}
The set of Brjuno numbers can be characterized as the set where $B(x)<\infty$. 
Indeed, there is a positive constant $C$ such that
\begin{equation}\label{eq:B-logq/q}
\left|B(x)-\sum\limits_{n=0}^{\infty}\dfrac{\log (Q_{n+1}(x))}{Q_n(x)}\right|<C.
\end{equation}

Yoccoz proved that
$B(x) = -\log x + x B(1/x)$ on $(0,1)\setminus\mathbb{Q}$ and $B(x)=B(x+1)$
thus
the set of Brjuno numbers is invariant under the modular group $\mathrm{PSL}_{2}(\mathbb{Z})$.
In his work on the linearizability of quadratic polynomials, Yoccoz showed that $\Upsilon(x):= B(x)+\log r(x)$ defined on the set of Brjuno numbers is uniformly bounded, where $r(x)$ is the convergence radius of the linearization of the quadratic polynomial $P_\lambda (z)=\lambda z+z^2$ with $\lambda = e^{2\pi x i}$ \cite{Yoc95,BC04}.
With the second author, Moussa and Yoccoz conjectured that $\Upsilon$ is $1/2$-H\"{o}lder continuous \cite{MMY97}, a fact 
prompted also by the numerical study in \cite{Mar90}.
It was later proved that the interpolation function $\Upsilon$ extends continuously to $\mathbb{R}$ \cite{BC06} and that 
if $\eta>1/2$, $\Upsilon$ cannot be $\eta$-Hölder continuous \cite{Che08}.
The restriction $\Upsilon$ to the set of the high type numbers is $1/2$-H\"{o}lder continuous \cite{CC15}, where the high type numbers are the real numbers whose nearest-integer continued fraction partial quotients are at least $N_0$, for some large fixed $N_0\ge 2$.

The interplay between the H\"older regularity properties and various actions of (subgroups of) the modular group, already 
at the heart of the works \cite{MMY97, MMY01, MMY06, LMNN10} will be our object of the investigation.
We will discuss the $L^p$ and the H\"older regularity properties of the difference between $B$ and the Brjuno functions associated with the by-excess, odd and even continued fractions regarded as classical continued fractions by Vall\'ee \cite{Val03,Val06}. 
These three Brjuno functions satisfy cocycle equations under actions of $\mathrm{PSL}_2(\mathbb{Z})$, the index $2$-, and an index $3$-congruence subgroups of level $2$, respectively, see Appendix~\ref{ap:cocycle} for a proof of this 
fact.
We will give further motivations for our research and state the main theorems in the following subsections.

\subsection{Brjuno functions associated with the by-excess continued fractions} 
Let $0\le \alpha\le 1$. The $\alpha$-continued fraction map is given by
\begin{equation*}\label{eq:alpha}
A_{\alpha} (x) = |1/x - \lfloor 1/x + 1 - \alpha \rfloor|~ \text{for} ~x\in(0,\bar{\alpha}],
\end{equation*}
where $\bar{\alpha} = \max\{\alpha, 1-\alpha\}$.
This one-parameter family of maps contains the the regular, the nearest-integer and the by-excess continued fractions as three particular cases, corresponding respectively to $\alpha=1,1/2,0$.
Nakada 
\footnote{In \cite{Nak81}, the $\alpha$-continued fraction map is defined by $f_\alpha(x) = |1/x| - \lfloor |1/x|+1-\alpha\rfloor$ for $x\not=0$, $x\in[\alpha-1,\alpha)$.} 
\cite{Nak81} investigated metrical properties of the $\alpha$-continued fractions and explicitly computed the absolutely continuous invariant measure for $\frac{1}{2}\le\alpha\le 1$.
The dependence on the parameter $\alpha$ of the metric entropy of the map $A_\alpha$ has been object of several studies in \cite{LM08, CMPT10, CT13}.

The \emph{$(\alpha,\nu)$-Brjuno function} is defined by 
\begin{equation}\label{eq:B_al}
B_{\alpha,\nu}(x) = - \sum_{n=0}^\infty (\beta_{\alpha,n-1}(x))^{\nu}\log A^n_\alpha(x) \text{ for } x\in\mathbb{R}\setminus\mathbb{Q},
\end{equation}
where $x_{\alpha,0} = |x-\lfloor x+1-\bar{\alpha}\rfloor|$
\footnote{We correct the initial setup of $x_{\alpha,0} = |x-\lfloor x+1-{\alpha}\rfloor|$ in \cite{LMNN10} to $x_{\alpha,0} = |x-\lfloor x+1-\bar{\alpha}\rfloor|$ since $x_{\alpha,0}$ should be $x$ for $x\in(0,\overline{\alpha}]$.}
 and
$\beta_{\alpha,n}(x) = x_0A_\alpha(x_0)\cdots A_\alpha^n(x_0)$, $n\ge 0$ and $\beta_{\alpha,-1}\equiv1$.
From the definition it follows easily that $B_{\alpha,\nu}$ satisfies
\begin{equation}\label{eq:fteq}
B_{\alpha,\nu}(x) = -\log x + x^\nu B_{\alpha,\nu} (1/x)  \text{ for } x\in(0,\bar{\alpha}){\setminus \mathbb{Q}},
\end{equation}
$B_{\alpha,\nu}(x) = B_{\alpha,\nu}(-x)$ for $x\in(0,1-\bar{\alpha})\setminus \mathbb{Q}$ and
$B_{\alpha,\nu}(x) = B_{\alpha,\nu}(x+1)$ for $x\in \mathbb{R}\setminus \mathbb{Q}.$
The difference $B-B_{\alpha,1}$ is uniformly bounded for all $\alpha\in[1/2,1]$ and, as we have recalled several times, for $\alpha=1/2$ (the nearest-integer continued fraction), the difference $B-B_{1/2, 1}$ is $1/2$-H\"{o}lder continuous.

These (generalized) Brjuno functions with a positive real exponent $\nu$, which will be called  \emph{$\nu$-Brjuno functions}. In our recent study with Petrykiewicz and Schindler \cite{LMPS}, we investigated the BMO property of the $\nu$-Brjuno functions\footnote{We called the $\nu$-Brjuno function the \emph{$k$-Brjuno function} in \cite{LMPS}.} associated with the $\alpha$-continued fractions motivated by the study of (quasi) modular forms in \cite{Pet14,Pet17}.

The $\alpha$-Brjuno functions for $\alpha\in [0,1/2)$ are explored in \cite{LMNN10}, where it's proved that $B_{\alpha,\nu}-B_{\alpha,\nu}$ is uniformly bounded for $\alpha\in(0,1/2)$.
Moreover, the authors studied the Brjuno function $B_{0,1}$ associated with the by-excess continued fraction ($\alpha=0$).
\renewcommand\thesubfigure{\arabic{figure}\alph{subfigure}} 
\captionsetup[figure]{labelformat=empty}
\captionsetup[subfigure]{labelfont = rm, textfont = normalfont, singlelinecheck = on} 
\begin{figure}
\begin{subfigure}{0.49\textwidth}
\centering
\includegraphics[width=7cm]{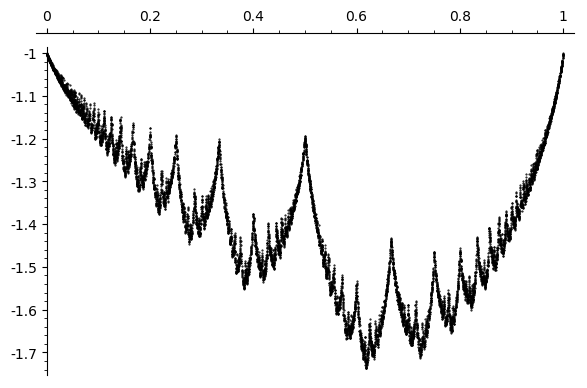}
\caption{The graph of $B(x) - [B_{0,1}(x)+B_{0,1}(-x)]$.}
\label{fig:dif.RB}
\end{subfigure}
\begin{subfigure}{0.49\textwidth}
\centering
\includegraphics[width=7cm]{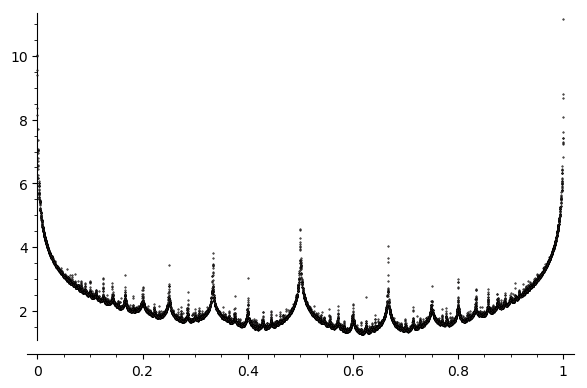}
\caption{The graph of $B_{odd,1}$.}
\label{fig:BO}
\end{subfigure}
\end{figure}
The $(0,1)$-Brjuno function $B_{0,1}$, which is called the \emph{semi-Brjuno function},  
has a different property: indeed one can prove that
\begin{equation}\label{eq:dif.B1}
B(x) - [B_{0,1}(x)+B_{0,1}(-x)]
\end{equation}
is uniformly bounded.
The corresponding continued fraction to $A_0$ is called the \emph{by-excess continued fraction} which has the expansion given by
$$x = \dfrac{1}{a^\ast_1-\dfrac{1}{\ddots-\dfrac{1}{a^\ast_n-\ddots}}} \quad \text{ for }x\in(0,1],$$
where $a^\ast_{n} = \lceil 1/A_0^{n-1}(x) \rceil$.
In \cite{LMNN10}, from a numerical computation of $B(x)-[B_{0,1}(x)+B_{0,1}(-x)]$, see also Figure~\subref{fig:dif.RB},
and by analogy with what had been proved for the difference between $B$ and $B_{1/2,1}$, it was conjectured that $B(x)-[B_{0,1}(x)+B_{0,1}(-x)]$ is $1/2$-H\"{o}lder continuous.

In this paper (see Section~\ref{sec:BF}) we will prove this conjecture. Moreover we will also prove an extension of this conjecture to all the $(0,\nu)$-Brjuno functions for $0<\nu\le 2$:

\begin{mainthm}\label{main1}
For $0<\nu\le 2$, the function $B_{1,\nu}(x)-\left[B_{0,\nu}(x)+B_{0,\nu}(-x)\right]$
can be extended from $\mathbb{R}\setminus \mathbb{Q}$ to $\mathbb{R}$ a $\nu/2$-H\"{o}lder continuous function.
\end{mainthm}

\subsection{Brjuno functions of the odd continued fraction and the even continued fraction}
Schweiger \cites{Sch82, Sch84} first investigated continued fractions with the odd partial quotients and the even partial quotients, which are called \emph{the odd continued fraction (OCF)} and \emph{the even continued fraction (ECF)}, respectively. 
We partition $[0,1]$ into $I_n:=\left[\frac{1}{n+1},\frac{1}{n}\right], \text{ for }n\in\mathbb{N}.$
The continued fraction maps of the ECF and the OCF are defined by
\begin{equation}\label{eq:AOE}
A_{odd}(x) = 
\begin{cases}
\frac{1}{x}- (2k-1), \text{ if }x\in I_{2k-1},\\
(2k+1) - \frac{1}{x}, \text{ if } x\in I_{2k},
\end{cases}
\text{and }
A_{even}(x) = 
 \begin{cases}
2k - \frac{1}{x},  \text{ if } x\in I_{2k-1},\\
\frac{1}{x} - 2k,  \text{ if } x\in I_{2k},
 \end{cases}
\end{equation}
for $k\in \mathbb{N}$, see Figure~\ref{fig:AO} and \subref{fig:AE} for their graphs.
Analogously to what we did for defining the  $\nu$-Brjuno function associated with the $\alpha$-continued fraction maps, we denote respectively with $B_{odd,\nu}$ and $B_{even,\nu}$
 the $\nu$-Brjuno functions associated with the OCF and the ECF, defined by
\begin{equation}\label{eq:BeBo}
B_{odd,\nu}(x) = - \sum_{n=0}^\infty \beta_{o,n-1}^\nu(x_{o,0})\log {x_{o,n}}
\quad\text{ and }\quad
B_{even,\nu}(x) = - \sum_{n=0}^\infty \beta_{e,n-1}^\nu(x_{e,0})\log {x_{e,n}},
\end{equation}
respectively, 
by setting $x_{o,0} = x_{e,0} = \|x\|_{2\mathbb{Z}}:=\min_{n\in\mathbb{Z}}\{x,2n\}$,
\begin{equation}\label{eq:set.x_n}
x_{o,n}:= A_{odd}^n(x_{o,0}), ~
 x_{e,n}:= A_{even}^n(x_{e,0}),
 \end{equation}
$ \beta_{o,n}(x) := \prod_{i=0}^n x_{o,i}$, and $\beta_{e,n}(x) :=\prod_{i=0}^n x_{e,i}, ~n\ge 0$
with $\beta_{o,-1}(x) = \beta_{e,-1}(x)=1$,
see Figure~\subref{fig:BO} and \ref{fig:EOO} for their graphs.
By definition, both $B_{odd,\nu}$ and $B_{even,\nu}$ are even $2\mathbb{Z}$-periodic functions, i.e. they both satisfy
\begin{equation}\label{eq:even2Z}
f(x) = f(-x)\text{ for all }x\in(0,1)\setminus\mathbb{Q} \text{ and }f(x)=f(x+2) \text{ for all }x\in\mathbb{R}\setminus \mathbb{Q}.
\end{equation}
\begin{figure}
$\begin{matrix}
\includegraphics[width=7cm]{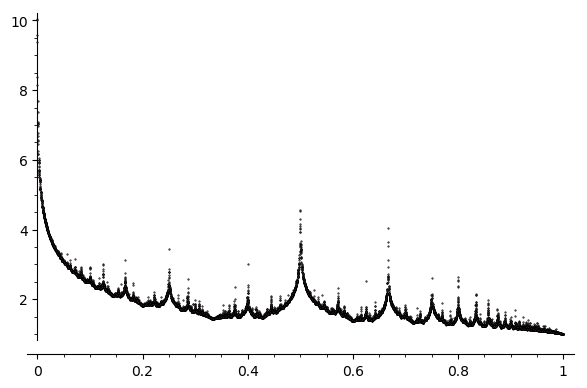} 
&
\includegraphics[width=7cm]{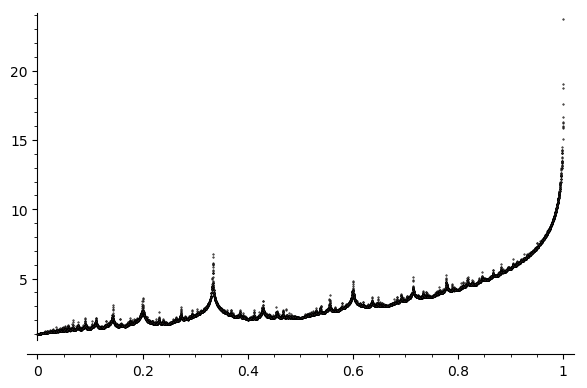}
\\
\includegraphics[width=7cm]{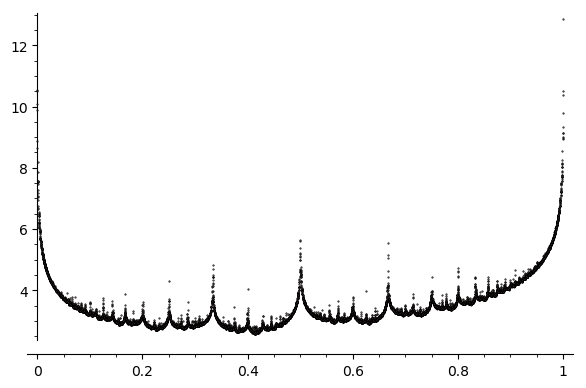} 
&
\includegraphics[width=7cm]{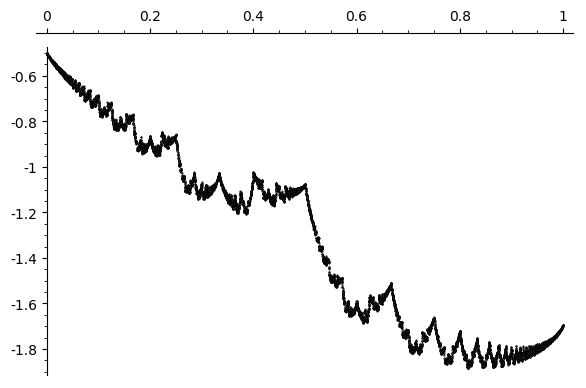}
\end{matrix}$
\caption{The graphs of $B_{even,1}$ (upper-left), $B_{oo,1}$ (upper-right), $B_{even,1}+\frac{1}{2}B_{oo,1}$ (lower-left) and $B-[B_{even,1}+\frac12 B_{oo,1}]$ (lower-right).}
\label{fig:EOO}
\end{figure}

It is well known that there is a very strict connection
between the RCF and the geodesic flow on the modular surface $\mathbb{H}/\mathrm{SL}_2(\mathbb{Z})$  \cite{Ser85}.
Boca and Merriman established connections of the OCF with the geodesic flow on $\mathbb{H}/\Gamma$, and of the ECF with that on $\mathbb{H}/\Theta$ \cite{BM18}, where $\Gamma$ and $\Theta$ are the two 
subgroups of the modular group given by
\begin{align}
&\label{eq:Gem} \Gamma =\left\{\begin{pmatrix}a&b\\c&d\end{pmatrix}\in\mathrm{SL}_2(\mathbb{Z}):\begin{pmatrix}a&b\\c&d\end{pmatrix}\equiv 
\begin{pmatrix}1&0\\0&1\end{pmatrix}, \begin{pmatrix}1&-1\\1&0\end{pmatrix}\text{ or } 
\begin{pmatrix}0&-1\\1&1\end{pmatrix} \text{ (mod 2)}\right\},\\
&\label{eq:The} \Theta =\left \{\begin{pmatrix}a&b\\c&d\end{pmatrix}\in\mathrm{SL}_2(\mathbb{Z}): \begin{pmatrix}a&b\\c&d\end{pmatrix}\equiv \begin{pmatrix}1&0\\0&1\end{pmatrix} \text{ or }\begin{pmatrix}0&-1\\1&0\end{pmatrix}\text{ (mod }2\text{)}\right\}.
\end{align}
One of our motivations for studying the OCF- and ECF-Brjuno functions is came from the study of regularity properties of automorphic forms under the action of congruence subgroups $\Gamma$ and $\Theta$ of level $2$, respectively.
The OCF- and ECF Brjuno functions satisfy \eqref{eq:even2Z} and
\begin{equation}\label{eq:fun.o}
B_{odd,\nu}(x) = - \log x + x^\nu B_{odd,\nu}\left(1-1/x\right) \text{ for all } x\in (0,1)\setminus\mathbb{Q},
\end{equation}
\begin{equation}\label{eq:fun.e}
B_{even,\nu}(x) = - \log x + x^\nu B_{even,\nu}(-1/x) \text{ for all } x\in (0,1)\setminus \mathbb{Q}.
\end{equation}
With these functional equations, the OCF-, ECF-Brjuno function can be interpreted as cocycles under the actions of $\Gamma\sqcup\left(\begin{smallmatrix}-1&0\\0&1\end{smallmatrix}\right)\Gamma$, $\Theta\sqcup\left(\begin{smallmatrix}-1&0\\0&1\end{smallmatrix}\right)\Theta$, respectively, see Appendix~\ref{ap:cocycle} for the details.

We prove a H\"older regularity property of $B_{odd,\nu}$  which shows how the difference between the 
classical Brjuno function and $B_{odd,\nu}$ is as regular as that between the classical Brjuno function and the semi-Brjuno function:
\begin{mainthm}\label{main2}
For $\nu>0$, the function $B_{1,\nu}(x) - B_{odd,\nu}(x)$ is uniformly bounded.
Moreover, for $0<\nu\le 2$, $B_{1,\nu}(x) - B_{odd,\nu}(x)$ can be extended to a $\nu/2$-H\"{o}lder continuous function on $\mathbb{R}$.
\end{mainthm}

A further motivation for the study of $B_{even,\nu}$ comes from the surprising role that this function plays in a very classical problem.
In the 1850s Riemann suggested that the function given by the  lacunary Fouries series $S(x)= \sum_{k=1}^\infty \frac{\sin(\pi k^2 x)}{k^2}$ was an example of a continuous but nowhere differentiable function. 
After the classical investigations of Hardy and Littlewood and especially Gerver's \cite{Ger71} and Jaffard's \cite{Jaf96} works, show that
$S$ is only differentiable at rational numbers of the reduced form with odd numerator and denominator.
Itatsu \cite{Ita81} 
reproved the non-differentability and differentiability on the rationals by using a relation between $S$ and the theta function $\theta(x) = \sum_{n\in \mathbb{Z}} \exp(i\pi n^2 x)$, which is an automorphic form of weight $1/2$ under the action of $\Theta$.

More recently, Revoal and Seuret  \cite{RS15} dealt with Dirichlet series 
$$F_{s}(x,t) = \sum_{k=1}^\infty \frac{\mathrm{exp}({i\pi k^2 x + 2\pi i k t})}{k^s}$$ for $s\in(1/2,1]$ as a generalization of $S$.
They proved that the convergence of $F_{s}$ depends on Brjuno-type conditions given by two formulas which relate $F_s(x,t)$ to the Gauss map $G$ and to the even continued fraction map $A_{even}$, see \cite[Theorem 3]{RS15} for the formulas.
The formula relating $F_s$ to $A_{even}$ is much simpler and one can see how the convergence of $F_s(x,t)$ is determined by the finiteness of the even Brjuno function with exponent $1/2$, that is,a 
the function $F_1(x)$ converges if $B_{even,1/2}(x)<\infty$ and $\sum_{n=0}^\infty (xA_{even}(x)\cdots A_{even}^{n-1}(x))^{1/2} < \infty$.
\footnote{
In \cite{RS15}, they stated that  
$F_1(x)$ converges if $\sum_{n=0}^\infty \left|xT(x)\cdots T^{n-1}(x)\right|^{1/2}\log\left(1/|T^{n}(x)|\right)<\infty$ and 
$\sum_{n=0}^\infty |xT(x)\cdots T^{n-1}(x)|^{1/2} < \infty,$ where $T:[-1,1]\setminus\{0\}\to [-1,1]$ defined by $x\mapsto -1/x \text{ mod }2$. 
The statements are equivalent since $|T(|x|)| = A_{even}(x)$. 
We note, moreover, that even the second condition could be reformulated in terms of the formalism of Brjuno function slightly generalizing what shown in \cite{MMY06}.
}

We will show that the $B_{1,\nu}(x)<\infty$ can be described by a combination of $B_{even,\nu}$ and a Brjuno-type function $B_{oo,\nu}$ associated with the odd-odd continued fraction (OOCF), see \eqref{eq:BOO} for the definition.
Kim, Lingmin and the first author \cite{KLL22} defined the OOCF algorithm to study Diophantine approximation properties related to the parities of principal convergents. 
\begin{mainthm}\label{main3}
For any $\nu>0$, the function
$B_{1,\nu}(x) -  \left[B_{even,\nu}(x) + \frac{1}{2}B_{oo,\nu}(x)\right]$
is uniformly bounded.
\end{mainthm}

A numerical computation of $B -  \left[B_{even,1} + \frac{1}{2}B_{oo,1}\right]$ leads to the graph in Figure~\ref{fig:EOO}: we conjecture that the function $x\mapsto B_{1,\nu}(x)-[B_{even,\nu}(x)+\frac{1}{2}B_{oo,\nu}(x)]$ extends to the whole real line as a $\mathbb{Z}$-periodic $1/2$-H\"{o}lder continuous function.

The paper is organized as follows.
In Section~\ref{sec:BF}, we generalize the H\"older continuity property of the Brjuno functions associated with the nearest-integer continued fraction to the $(1/2,\nu)$-Brjuno functions and prove Theorem~\ref{main1}.
In Section~\ref{sec:OF}, we introduce the Brjuno function associated with the OCF and prove Theorem~\ref{main2}.
In Section~\ref{sec:EF}, we define the Brjuno function associated with the ECF and to the OOCF and prove Theorem~\ref{main3}.

\section{The semi-Brjuno function}\label{sec:BF}
In this section, we study the semi-Brjuno function $B_{0,\nu}$ defined in \cite{LMNN10} as in \eqref{eq:B_al} with $\alpha=0$ associated to the by-excess continued fraction and we give a proof of Theorem~\ref{main1}.

Let $\alpha\in [0,1]$. Let us consider the linear space of real valued measurable functions $f$ such that
\begin{equation}\label{eq:ft.cond.al}
f(x) = f(-x) \text{ for } x\in (\overline{\alpha}-1,0)\setminus\mathbb{Q}\quad \text{ and }\quad f(x) = f(x+1)\text{ for }x\in\mathbb{R}\setminus\mathbb{Q}.
\end{equation}
On the functions satisfying \eqref{eq:ft.cond.al}, for $\alpha\in [0,1]$ and $\nu\ge 0$, we consider the linear operator $T_{\alpha,\nu}$ is defined by 
\begin{equation}\label{eq:op.alpha}
T_{\alpha,\nu} f(x) = x^\nu f(1/x) \quad \text{ for }x\in (0,\bar{\alpha}),
\end{equation}
where $T_{\alpha,\nu}f$ is extended to the whole real line by requiring that it satisfies \eqref{eq:ft.cond.al}, i.e. $T_{\alpha,\nu}f(x) = x^\nu f(A_{\alpha}(x))$.
Then we have the functional equation
$$[(1-T_{\alpha,\nu})B_{\alpha,\nu}](x) = -\log(|x-\lfloor x-\alpha+1\rfloor|),$$
where the function on the right hand side satisfies \eqref{eq:ft.cond.al} and the value is $-\log{x}$ for $x\in(0,\bar{\alpha})$.

Let us consider a function $g$ defined on an interval $I\subset \mathbb{R}$.
For $0<\eta\le 1$, we define \emph{the H\"{o}lder's $\eta$-seminorm} by
$$|g|_\eta := \sup_{x < x' \in I}\frac{|g(x)-g(x')|}{|x-x'|^\eta}.$$
We define the norm $\|g\|_\eta := |g|_\eta + \|g\|_{\mathcal{C}^0}.$
We say that a uniformly bounded function $g$ is \emph{$\eta$-H\"{o}lder continuous} if $\|g\|_\eta <\infty$ and denote by $\mathcal{C}^\eta_{I}$ the set of the $\eta$-H\"{o}lder continuous functions.
Note that $\mathcal{C}^1_I$ is the space of the Lipschitz continuous functions and it is different from the space of continuously differentiable functions, usually denoted by $C^1$.

From now on, we recall some properties of the operator $T_{1/2,\nu}$ in \cite[Thm 4.4 and 4.6]{MMY97}.
The spectral radius of $T_{1/2,\nu}$ is strictly less than $1$ on the space of functions in $L^p$ satisfying \eqref{eq:ft.cond.al} for $\alpha=1/2$ and $p\ge 1$.
Then the operator ${\bf B_{1/2,\nu}} : = (1-T_{1/2,\nu})^{-1}$ is well-defined.
\begin{prop} \label{MMY4.4} If $f\in \mathcal{C}^\eta_{[0,1/2]}$ for $\nu<2\eta$, then we have ${\bf B_{1/2,\nu}} f \in \mathcal{C}^{\nu/2}_{[0,1/2]}.$
\end{prop}

From the above proposition, it is shown that the difference $B_{1,1}-B_{1/2,1}$ is $1/2$-H\"older continuous.
It can be generalized to the Brjuno functions with the exponents $0<\nu\le2$.
\begin{thm}\label{MMY4.6}
For $0<\nu\le2$, the difference $B_{1,\nu} - B_{1/2,\nu}$ can be extended from $\mathbb{R}\setminus\mathbb{Q}$ to $\mathbb{R}$ as a $\nu/2$-H\"{o}lder continuous $\mathbb{Z}$-periodic function.
\end{thm}
\begin{proof}
Let us denote the even part and the odd part of $B_{1,\nu}$ by $B^+_{1,\nu}$ and $B^-_{1,\nu}$, respectively.
For $x\in[0,1/2]$, we have
\begin{align*}
B^-_{1,\nu}(x)  & = \frac{1-(1-x)^\nu}{2}\log \frac{1-x}{x}\in \mathcal{C}_{[0,1/2]}^{\nu'} \text{ for }0\le \nu'<\nu, \text{ and } \\
B^+_{1,\nu}(x) & = -\log x - B^-_{1,\nu}(x) + x^\nu B^-_{1,\nu}(1/x) + x^\nu B^+_{1,\nu}(1/x).
\end{align*}
Let $\Delta(x) = B^+_{1,\nu}(x)-B_{1/2,\nu}(x)$. Then we have $\Delta(x) = - B_{1,\nu}^-(x) + x^\nu B_{1,\nu}^-(1/x) + x^\nu \Delta(1/x)$. 
Then, $\Delta(x) = - {\bf B_{1/2,\nu}} B_{1,\nu}^-(x) + \sum_{n=0}^\infty T_{1/2,\nu}^n(x^\nu B_{1,\nu}^-(1/x))$.
By letting $\veps_n(x) = \mathrm{sgn}(\frac{1}{A_{1/2}^n(x)}-\lfloor \frac{1}{A_{1/2}^n(x)}+\frac{1}{2}\rfloor)$, each term of the second summand is
\begin{align*}
T_{1/2,\nu}^n(x^\nu B_{1,\nu}^-(1/x)) &= \beta_{1/2,n-1}^\nu (x) (A_{1/2}^n(x))^\nu B_{1,\nu}^-\left(\frac{1}{A_{1/2}^n(x)}\right) \\
&=  (\beta_{1/2,n}(x))^\nu B_{1,\nu}^-(\veps_n(x) A_{1/2}^{n+1}(x)) = \veps_n(x) (\beta_{1/2,n}(x))^\nu B^-_{1,\nu}(A_{1/2}^{n+1}(x)).
\end{align*}
Therefore, $\Delta = - {\bf B_{1/2,\nu}} B_{1,\nu}^- + \sum_{n=1}^\infty \veps_n T_{1/2,\nu}^nB_{1,\nu}^-$.
It is shown that $ \sum_{n=1}^\infty \veps_n T_{1/2,\nu}^n f \in \mathcal{C}^{\nu/2}_{[0,1/2]}$ if $f\in \mathcal{C}^\eta_{[0,1/2]}$ for $\nu<2\eta$ in \cite{MMY97}.
Combining with Proposition~\ref{MMY4.4}, we have $\Delta \in \mathcal{C}^{\nu/2}_{[0,1/2]}$.
\end{proof}

We recall the functional equation in \eqref{eq:fteq} for $\alpha=0$ and $\alpha=1/2$ as
\begin{align}
& \label{eq:ft.eq.B} B_{0,\nu}(x) = - \log x + x^\nu B_{0,\nu} (- x^{-1}) \text{ for }x\in (0,1)\setminus\mathbb{Q}, \\
& \label{eq:ft.eq.N} B_{1/2,\nu}(x) = - \log x + x^\nu B_{1/2,\nu}(x^{-1}) \text{ for }x\in (0,1/2)\setminus\mathbb{Q}.
\end{align}
\begin{proof}[Proof of Theorem~\ref{main1}]
Let $\nu\in (0,2]$.
By Theorem~\ref{MMY4.6} and $B_{0,\nu}(x)=B_{0,\nu}(x+1)$, it is sufficient to show that $\Delta(x):= B_{1/2,\nu}(x) - [B_{0,\nu}(x) + B_{0,\nu}(1-x)]$ is ${\nu}/2$-H\"{o}lder continuous.
Since $\Delta$ is an even $\mathbb{Z}$-periodic function, it is enough to show that $\Delta\in\mathcal{C}^{\nu/2}_{[0,1/2]}$.

For $x\in[0,1/2]$, $A_0(x) = A_{1/2}(x)$ if $A_0(x)<1/2$, otherwise, $A_0(x) = 1-A_{1/2}(x)$.
From the functional equations in \eqref{eq:ft.eq.B} and \eqref{eq:ft.eq.N}, we have
\begin{align}
\Phi(x)
&\nonumber := \Delta(x) -  x^\nu \Delta\big(A_{1/2}(x)\big) \\
&\nonumber = - x^\nu B_{0,\nu}\big(A_0(x)\big) - B_{0,\nu}(1-x) + x^\nu \left[B_{0,\nu}\big(A_{1/2}(x)\big) + B_{0,\nu}\big(1-A_{1/2}(x)\big)\right] \\
&\label{eq:rest.B} = x^\nu B_{0,\nu} \big(1-A_0(x)\big)-B_{0,\nu} (1-x),
\end{align}
which means that $\Delta = {\bf B_{1/2,\nu}} \Phi.$
See Figure~\ref{fig:rest.B} for the graphs of $\Phi$ with $\nu=1$ and $\nu=1/2$.
%%%%%%%%%%%%%%%%%%%%%%%%%%%%%%%%%%%
\begin{figure}
$\begin{matrix}
\includegraphics[width=8cm]{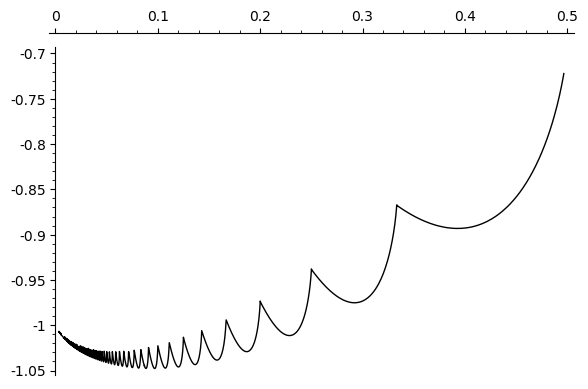} &
\includegraphics[width=8cm]{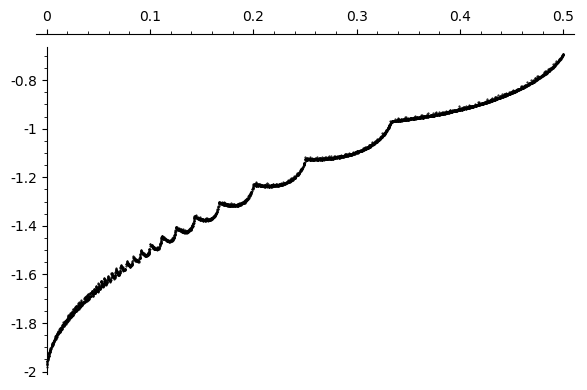} 
\end{matrix}$
\caption{The graphs of $xB_{0,1}(1-A_{0}(x))-B_{0,1}(1-x)$ (left) and $x^{1/2}B_{0,1/2}(1-A_0(x))-B_{0,1/2}(1-x)$ (right).}
\label{fig:rest.B}
\end{figure}
%%%%%%%%%%%%%%%%%%%%%%%%%%%%%%%%%%%
By Proposition~\ref{MMY4.4}, it is enough to show that $\Phi\in \mathcal{C}^\eta_{[0,1/2]}$ for $\eta>\nu/2$.
We will prove it by the following two steps:
\begin{enumerate}
\item\label{it:thm1-1}
We will first show that $\Phi(x)$ on $(0,1/2]$ can be represented by
\begin{equation}\label{eq:rest.B2}
\begin{split}
\Phi(x) 
 = & ((1-(n-1)x)^\nu - x^\nu) \log(1-nx) + x^\nu \log x \\
    & + \sum_{k=1}^{n-1}((1-(k-1)x)^\nu - (1-kx)^\nu) \log(1-kx), \quad \text{ where } n = \left\lfloor \frac1 x\right\rfloor.
\end{split}
\end{equation}

\item\label{it:thm1-2} Then, we will show that $\Phi(x)-x^\nu \log x$
is of $\mathcal{C}^{\eta}_{[0,1/2]}$ for $\eta>\nu/2$.
Since $x^\nu \log x$ is $\eta$-H\"{o}lder continuous for $\eta\in(0,\nu)$, it completes the proof.
\end{enumerate}

To show \eqref{it:thm1-1}, we introduce the \emph{jump transformation} $F_{jump}$ of $A_{0}$ with respect to the interval $[0,1/2)$ which is defined by
$$F_{jump}(x) = A_{0}^{m(x)+1}(x),$$
where $m(x) = \min\{n\ge0: A_{0}^{n}\in[0,1/2)\}$. 
Let us denote by
$$
 \psi(x) := - \sum_{k=0}^{m(x)} \big(kx-(k-1)\big)^\nu \log\left(\frac{(k+1)x-k}{kx-(k-1)}\right),
$$
which is equal to $B_{0,\nu}(x) - (xA_0(x)\cdots A_0^{m(x)}(x))^\nu B_{0,\nu}(F_{jump}(x))$.
Since $A_{0}+G \equiv 1$, we have
$$
x^\nu B_{0,\nu}(1-A_0(x)) = (xG(x))^{\nu} B_{0,\nu} (1-G^2(x)) - x^\nu \log(G(x)). 
$$
Let us assume that $x\in\left[\frac{1}{n+1},\frac{1}{n}\right]$ for some $n\ge2$.
Then, $x G(x) = 1-nx$.
Since $1-x\in\left[\frac{n-1}{n}, \frac{n}{n+1}\right]$, we have $m(1-x) = n-1$.
Thus $(1-x)A_0(1-x)\cdots A_0^{m(1-x)}(1-x) = 1-nx$.
Kraaikamp and Nakada \cite{KN01} observed that
$1-G^2(x) = F_{jump}(1-x).$
Thus, we have $x^\nu B_{0,\nu} (1-A_0(x))-B_{0,\nu} (1-x) = - x^\nu \log(G(x))-\psi(1-x).$
Since 
$x^\nu \log(G(x)) = x^\nu \log(1/x-n)$
and
\begin{align*}
- \psi(1-x) 
& = \sum_{k=0}^{n-1}(1-kx)^\nu \log\left(\frac{1-(k+1)x}{1-kx}\right) \\
& = \sum_{k=1}^{n-1}((1-(k-1)x)^\nu - (1-kx)^\nu) \log(1-kx) + (1-(n-1)x)^\nu \log(1-nx),
\end{align*}
we obtain \eqref{eq:rest.B2}.

We will show \eqref{it:thm1-2} as follows.
Let $h(x): = ((1-(n-1)x)^\nu-x^\nu)\log(1-nx)$ and 
$$g(x): = \sum_{k=1}^{n-1} ((1-(k-1)x)^\nu - (1-kx)^\nu)\log(1-kx).$$
Then, $h(x)+g(x) = \Phi(x)-x^\nu \log x$.

We have
$$\lim_{x\nearrow \frac 1n} h(x)
= \lim_{x\nearrow \frac 1n}\frac{(1-nx+x)^\nu - x^\nu}{1-nx}(1-nx)\log(1-nx)=0.$$ 
Thus, $h(x)+g(x)$ is continuous on $(0,1/2]$.
Given $x\in [0,1/2]$, let $n=\lfloor 1/x \rfloor$.
Note that $x\to 0$ is equivalent to $n\to \infty$.
Since $0\le 1-nx \le \frac{1}{n+1}$, $\lim_{x\to 0} (1-nx) = 0$.
Thus we have $\lim_{x\to0} h(x) = 0.$
Since the maximal length of the subintervals in the partition $\{1-kx: k= 0,\cdots, n\}$ approaches $0$ when $x\to 0$, by using the Riemann–Stieltjes integral, we have
\begin{align*}
& \lim_{x\to 0} g(x)
 = \int_0^1 \nu y^{\nu-1}\log y \mathrm{d}y 
 = \left[ y^\nu \log y - \frac{1}{\nu}y^\nu\right]^1_0 =  -\frac{1}{\nu}. 
 \end{align*}
Also we have 
\begin{equation}\label{eq:rest.B3}
g(x)
\sim \int_{1-nx}^1 \nu y^{\nu-1}\log y \mathrm{d}y = - (1-nx)^\nu\log(1-nx) + (1-nx)^\nu/\nu - 1/\nu 
\end{equation}
as $x\to 0$.

Let us define $h(0)=0$ and $g(0)=-1/\nu$. Then $h+g$ is continuous and uniformly bounded on $[0,1/2]$.
We will complete the proof by showing that for some $\eta>\nu/2$, for all $x'\in[0,1/2]$, there exists
$$\lim_{x\to x'}\frac{|(h+g)(x)-(h+g)(x')|}{|x-x'|^\eta}.$$

The functions $h, g$ are differentiable on $(0,1/2]\setminus \bigcup_{n=2}^\infty\{1/n\}$.
The function $g$ is right and left differentiable on $\frac{1}{n+1}$.
The function $h$ is right differentiable on $\frac{1}{n+1}$. 
If $\eta\in(\nu/2, \min\{\nu,1\})$, then
$$\lim_{x\nearrow \frac 1n}\frac{|h(x)|}{|x-1/n|^\eta} 
= \lim_{x\nearrow \frac 1n} \frac{|(1-nx+x)^\nu - x^\nu|}{n^{-\eta}|1-nx|}|1-nx|^{1-\eta}\log\frac{1}{1-nx} = 0.$$
Thus, for $x\in(0,1/2]$, we have
$\lim_{x'\to x}\frac{|(h+g)(x)-(h+g)(x')|}{|x-x'|^\eta} = 0.$

For $x\in \left[\frac{1}{n+1},\frac{1}{n}\right]$, we have $0\le 1-nx\le x$.
Thus we have
\begin{align*}
 \lim_{x\to 0}\frac{|h(x)|}{x^\eta}  \le \lim_{x\to 0} \frac{|h(x)|}{|1-nx|^\eta} = 0.
\end{align*}
By \eqref{eq:rest.B3}, we have
\begin{align*}
\lim_{x\to 0}\frac{|g(x)+1/\nu|}{x^\eta} \lesssim 
\lim_{x\to 0}(1-nx)^{\nu-\eta}|\log (1-nx) - 1/\nu| = 0.
\end{align*}
Thus $\lim_{x\to 0}\frac{|(h+g)(x)+1/\nu|}{x^\eta} = 0$ and then $h+g \in \mathcal{C}^{\eta}_{[0,1/2]}$ for $\eta\in(\nu/2,\min\{\nu,1\})$.
\end{proof}

\section{The OCF-Brjuno function}\label{sec:OF}
In this section, we recall some known arithmetic and measure-theoretical properties of the OCF.
Then we prove Theorem~\ref{main2} giving the H\"older regularity of the $\nu$-Brjuno functions associated to the OCF. 
\begin{figure}
\begin{tikzpicture}[scale=5]
\draw (1 ,0) -- (1 +1,0);
\draw (1 +0,0) -- (1 +0,1);
\draw[domain=1 +1/2:1 +1,black] plot (\x,{(1/(\x-1 )-1)});
\draw[domain=1 +1/3:1 +1/2,black] plot (\x,{-(1/(\x-1 )-3)});
\draw[domain=1 +1/4:1 +1/3,black] plot (\x,{(1/(\x-1 )-3)});
\draw[domain=1 +1/5:1 +1/4,black] plot (\x,{-(1/(\x-1 )-5)});
\draw[domain=1 +1/6:1 +1/5,black] plot (\x,{(1/(\x-1 )-5)});
\draw[domain=1 +1/7:1 +1/6,black] plot (\x,{-(1/(\x-1 )-7)});
\draw[domain=1 +1/8:1 +1/7,black] plot (\x,{(1/(\x-1 )-7)});
\draw[domain=1 +1/9:1 +1/8,black] plot (\x,{-(1/(\x-1 )-9)});
\draw[domain=1 +1/10:1 +1/9,black] plot (\x,{(1/(\x-1 )-9)});
\draw[domain=1 +1/11:1 +1/10,black] plot (\x,{-(1/(\x-1 )-11)});
\draw[domain=1 +1/12:1 +1/11,black] plot (\x,{(1/(\x-1 )-11)});
\draw[domain=1 +1/13:1 +1/12,black] plot (\x,{-(1/(\x-1 )-13)});
\draw[domain=1 +1/14:1 +1/13,black] plot (\x,{(1/(\x-1 )-13)});
\draw[domain=1 +1/15:1 +1/14,black] plot (\x,{-(1/(\x-1 )-15)});

\draw[domain=1 +1/16:1 +1/15,black] plot (\x,{(1/(\x-1 )-15)});
\draw[domain=1 +1/17:1 +1/16,black] plot (\x,{-(1/(\x-1 )-17)});
\draw[domain=1 +1/18:1 +1/17,black] plot (\x,{(1/(\x-1 )-17)});
\draw[domain=1 +1/19:1 +1/18,black] plot (\x,{-(1/(\x-1 )-19)});
\draw[domain=1 +1/20:1 +1/19,black] plot (\x,{(1/(\x-1 )-19)});
\draw[domain=1 +1/21:1 +1/20,black] plot (\x,{-(1/(\x-1 )-21)});

\draw[dotted, thin] (1 +1/2,0) -- (1 +1/2,1);

\node at (1 -0.05,1){\tiny$1$};
\node at (1 +1,-0.08){\tiny$1$};
\node at (1 -0.05,-0.08){\tiny$0$};
\node at (1 +1/2,-0.1) {\small$\frac12$};

\draw[dotted, thin] (1 +1/3,0) -- (1 +1/3,1);

\node at (1 +1/3,-0.1) {\small$\frac13$};
\node at (1 +1/4,-0.1) {\small$\frac14$};
\node at (1 +1/5,-0.1) {\small$\frac15$};

\draw[dotted, thin](1 +1/4,0)--(1 +1/4,1);
\draw[dotted, thin](1 +1/5,0)--(1 +1/5,1);

\end{tikzpicture}
\caption{The graph of $A_{odd}$.}\label{fig:AO}
\end{figure}
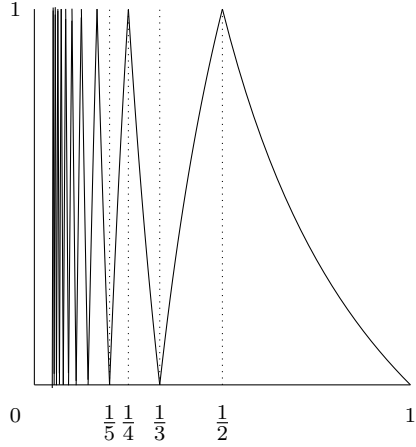
From now on, let us denote the reciprocal of the golden mean by $g=\frac{\sqrt{5}-1}{2}$.
The OCF map $A_{odd}$ defined as \eqref{eq:AOE} is ergodic with respect to a probability measure
$$dm_{odd}(x):=\frac{1}{3\log (g^{-1})}\left(\frac{1}{g^{-1}-1+x}+\frac{1}{g^{-1}+1-x}\right)dx,$$
see \cite[Thm 1]{Sch82}.
We denote by $\lfloor t \rfloor_{odd} := 2\lfloor t/2 \rfloor +1$ a nearest odd integer of $t$.
For $x\in\mathbb{R}$, the \emph{OCF partial quotients} $(a_{o,n},\veps_{o, n})$, $n\ge0$ are defined by 
\begin{equation}\label{eq:digitOE}
a_{o,0} := \lfloor x \rfloor_{odd}, ~\veps_{o,0} :=  \mathrm{sign}(x-a_{o,0}),
~a_{o,n} := \left\lfloor \frac{1}{x_{o,n-1}}\right\rfloor_{odd} \text{ and }
\veps_{o,n} := \text{sign}\bigg(\dfrac{1}{x_{o,n-1}}- a_{o,n}\bigg),
\end{equation}
where $x_{o,n}$ is an $A_{odd}$-orbit of $x$ defined as in \eqref{eq:set.x_n}.
Every $x\in \mathbb{R}$ has the following OCF-expansion:
\begin{equation*}
 x =a_{o,0} + \dfrac{\veps_{o,0}}{a_{o,1}+\dfrac{\veps_{o,1}} {\ddots+\dfrac{\ddots}{a_{o,n}+ \dfrac{\veps_{o,n}}{a_{o,n+1}+\ddots}}}}.
\end{equation*}
The \emph{$n$th principal convergent of  {the} OCF} is the truncated continued fraction expansion denoted by
$$
\frac{p_{o,n}}{q_{o,n}} := a_{o,0} + \dfrac{\veps_{o,0}}{a_{o,1}+\dfrac{\veps_{o,1}}{a_{o,2}+\ddots+\dfrac{\veps_{o,n-1}}{a_{o,n}}}}.
$$
Let $p_{o,-1}=1$, $q_{o,-1}=0$, $p_{o,0}=a_{o,0}$ and $q_{o,0}=1$.
We recall the following basic properties of $p_{o,n}/q_{o,n}$.
\begin{rmk}\label{rk:p/qO}
Since we have
$$
\begin{pmatrix}p_{o,n} & p_{o,n-1} \\ q_{o,n} & q_{o,n-1} \end{pmatrix} = 
\begin{pmatrix}a_{o,0} & 1\\ 1 & 0 \end{pmatrix}
\begin{pmatrix}a_{o,1} & 1\\ \veps_{o,0} & 1 \end{pmatrix}
\begin{pmatrix}a_{o,2} & 1\\ \veps_{o,1} & 1 \end{pmatrix}
\cdots
\begin{pmatrix}a_{o,n} & 1 \\ \veps_{o,n-1} & 1 \end{pmatrix},
$$
the following recursive formulas holds:
$$p_{o,n} = a_{o,n}p_{o,n-1}+\veps_{o,n-1}p_{o,n-2} \quad \text{and}\quad q_{o,n} = a_{o,n}q_{o,n-1}+\veps_{o,n-1}q_{o,n-2} \quad \text{ for } n\ge 1,$$
and
\begin{equation}\label{eq:det.pq} 
p_{o,n}q_{o,n-1} - p_{o,n-1}q_{o,n} = (-1)^{n+1}\veps_{o,0}\veps_{o,1}\veps_{o,2}\cdots \veps_{o,n-1}.
\end{equation}
By using $a_{o,n}+\veps_{o,n} x_{o,n} = 1/x_{o,n-1}$, we have
\begin{equation}\label{eq:x_n} 
x = \cfrac{p_{o,n}+p_{o,n-1}\veps_{o,n}x_{o,n}}{q_{o,n}+q_{o,n-1}\veps_{o,n}x_{o,n}} \quad\text{and}\quad
x_{o,n} = -\veps_{o,n} \frac{q_{o,n}x-p_{o,n}}{q_{o,n-1}x- p_{o,n-1}}.
\end{equation}
Thus we have
\begin{equation}\label{eq:qx-p}
\beta_{o,n} =  x_{o,0} x_{o,1}x_{o,2}\cdots x_{o,n} = (q_{o,n} x - p_{o,n})  \prod_{i=0}^n (-\veps_{o,i}) 
= \cfrac{1}{q_{o,n+1}+\veps_{o,n+1}q_{o,n}x_{o,n+1}}.
\end{equation}
\end{rmk}

Note that $q_{o,n}$ can be less than $q_{o,n-1}$ if $a_{o,n}=1$ and $\veps_{o,n-1}=-1$.
We note however that for all $n$ we have 
\begin{align}
 \label{eq:q.O1} & \veps_{o,n}q_{o,n-1}q_{o,n}^{-1}  > g^{-1}-2, \\
 \label{eq:q.O2} & q_{o,n}  > g q_{o,n-1}.
\end{align}
Thus we have 
\begin{equation}\label{eq:q.O3} g^{-1}-2<\veps_{o,n} \frac{q_{o,n-1}}{q_{o,n}}<g^{-1}.\end{equation}
Combining \eqref{eq:qx-p} with \eqref{eq:q.O1}, we have 
\begin{equation}\label{eq:beq}
1-g \le \beta_{o,n}(x)q_{o,n+1}\le g^{-1}.
\end{equation}
If $x_{o,k} > g$ for some $0\le k \le n$, then $x_{o,k+1}={(x_{o,k})}^{-1}-1$.
Since $x_{o,k}x_{o,k+1} = 1 - x_{o,k} < 1 - g = g^2$, we have
\begin{equation}\label{eq:beO}
\beta_{o,n}\le g^{n} \quad \text{ for }n\ge 1.
\end{equation}

\begin{rmk}\label{rk:sum.1/q}
Let $\nu>0$.
From \eqref{eq:beq} and \eqref{eq:beO}, we have $q_{o,n}\ge \frac{1-g}{g^{n-1}}$, thus
\begin{equation*}
\sum_{n=0}^\infty \frac{1}{q_{o,n}^\nu}\le \sum_{n=0}^\infty \frac{g^{\nu(n-1)}}{(1-g)^\nu} =: c_{1,\nu}.
\end{equation*}
Since $\log q_{o,n} \lesssim_\nu q_{o,n}^{\nu/2}$.
Thus, there exists a positive constant $c_\nu$ such that 
\begin{equation*}
\sum_{n=0}^\infty \frac{\log q_{o,n}}{q_{o,n}^\nu}\le  {c_\nu} \sum_{n=0}^\infty \frac{1}{q_{o,n}^{\nu/2}} \le  {c_\nu} \sum_{n=0}^\infty\frac{g^{\nu(n-1)/2}}{(1-g)^{\nu/2}} =: c_{2,\nu}.
\end{equation*}
Similarly, since $Q_n \ge 2^{(n-1)/2}$, there are constants $C_1,C_2>0$ such that
\begin{equation}\label{eq:sum.1/qR}
\sum_{n=0}^\infty Q_n^{-\nu} \le C_{1,\nu} \quad \text{ and } \quad \sum_{n=0}^\infty Q_n^{-\nu}\log Q_n \le C_{2,\nu}.
\end{equation}
\end{rmk}

\subsection{Uniform boundedness}
In this subsection, we show the first assertion of Theorem~\ref{main2}.
From \eqref{eq:B-logq/q}, the following two propositions imply $B_{1,\nu}-B_{odd,\nu}\in L^\infty$.

\begin{prop}\label{pr:logq/qO}
There exists $C_\nu>0$ such that for all $x\in\mathbb{R}\backslash\mathbb{Q}$, 
$$\left|\sum_{n=0}^\infty \frac{\log Q_{n+1}}{Q_n^\nu} - \sum_{n=0}^\infty\frac{\log q_{o,n+1}}{q_{o,n}^\nu}\right|\le C_\nu,$$
where $P_n/Q_n$ is the principal convergent of the RCF.
\end{prop}

\begin{prop}\label{pr:logq/qO2}
There exists $C_\nu>0$ such that for all $x\in\mathbb{R}\backslash \mathbb{Q}$,
$$\left|B_{odd,\nu}(x)-\sum_{n=0}^\infty\frac{\log q_{o,n+1}}{q_{o,n}^\nu}\right|\le C_\nu.$$
\end{prop}

Hartono and Kraaikamp \cite{HK02} described a relation between the RCF and the OCF and gave an algorithm for obtaining the OCF-expansion from the RCF-expansion.
By their result,  {we have} the following lemma immediately, see also \cite[Section 2.1 and Theorem 2.1]{Har03} and \cite{Kra91}.

\begin{lem}\label{le:p/q.O}
\begin{enumerate}
\item\label{it:leO1} If $p_{o,k}/q_{o,k} = P_n/Q_n$ for some $n$, then $p_{o,k+1}/q_{o,k+1} = P_{n+1}/Q_{n+1}$ or $(P_{n+1}+P_n)/(Q_{n+1}+Q_n)$.
\item\label{it:leO2} If $p_{o,k}/q_{o,k}$ is not a RCF convergent, then $p_{o,k}/q_{o,k} = (P_{n+1}+P_{n})/(Q_{n+1}+Q_n)$ and $p_{o,k+1}/q_{o,k+1} = P_{n+1}/Q_{n+1}$ for some $n$.
\end{enumerate}
\end{lem}

From the above lemma, 
$\{p_k^o/q_k^o\}$ is a subsequence of
\begin{equation}\label{eq:seq.o}
\frac{P_{-1}}{Q_{-1}},~ \frac{P_{-1}+P_0}{Q_{-1}+Q_0},~\frac{P_0}{Q_0},~ \frac{P_0+P_1}{Q_0+Q_1},~ \frac{P_1}{Q_1},\cdots,\frac{P_{n}}{Q_{n}},~\frac{P_{n}+P_{n+1}}{Q_{n}+Q_{n+1}},~\frac{P_{n+1}}{Q_{n+1}},\cdots
\end{equation}
containing all $P_n/Q_n$ and containing some $(P_{n}+P_{n+1})/(Q_{n}+Q_{n+1})$.
Note that the sequence as in \eqref{eq:seq.o} is not in order of denominators.
By using the above lemma, we can show that the $\nu$-Brjuno condition  {of the RCF} is equivalent to $\sum_{n=0}^\infty (\log q_{o,n+1})/ {q_{o,n}^\nu}<\infty$.

\begin{proof}[Proof of Proposition~\ref{pr:logq/qO}]
By Lemma~\ref{le:p/q.O}, 
if $p_{k}/q_{k}$ is a RCF convergent $P_n/Q_n$, then we have
$$\left|\frac{\log q_{o,k+1}}{(q_{o,k})^\nu} - \frac{\log Q_{n+1}}{Q_n^\nu}\right|\le \frac{\log 2}{Q_n^\nu}.$$
If $p_{o,k}/q_{o,k}$ is not a RCF convergent, then we have
$$ \frac{\log q_{o,k+1}}{(q_{o,k})^\nu} \le \frac{\log Q_{n+1}}{Q_{n+1}^\nu}.$$
Combining with Remark~\ref{rk:sum.1/q}, we have
$$\left|\sum_{k=0}^\infty \frac{\log q_{o,k+1}}{q_{o,k}^\nu} - \sum_{n=0}^\infty \frac{\log Q_{n+1}}{Q_n^\nu}\right| \le \sum_{n=0}^\infty \frac{\log Q_n + \log 2}{Q_n^\nu} \le C_{1,\nu} \log 2 + C_{2,\nu}.$$
\end{proof}

\begin{proof}[Proof of Proposition~\ref{pr:logq/qO2}]
We have
\begin{equation*}\begin{split}
&  - B_{odd,\nu}(x) + \sum_{n=0}^\infty \frac{\log q_{o,n+1}}{q_{o,n}^\nu} 
 = \sum_{n=0}^\infty \beta_{o,n-1}^\nu \log \frac{\beta_{o,n}}{\beta_{o,n-1}} + \sum_{n=0}^\infty \frac{\log q_{o,n+1}}{q_{o,n}^\nu}  \\
& = \sum_{n=0}^\infty \beta_{o,n-1}^\nu \log \beta_{o,n} q_{o,n+1} 
- \sum_{n=0}^\infty \beta_{o,n-1}^\nu \log \beta_{o,n-1} 
+ \sum_{n=0}^\infty \left(\frac{1}{q_{o,n}^\nu} - \beta_{o,n-1}^\nu\right) \log q_{o,n+1}.
\end{split}\end{equation*}
By \eqref{eq:beq} and Remark~\ref{rk:sum.1/q}, we have
$$\left|  \sum_{n=0}^\infty \beta_{o,n-1}^\nu\log \beta_{o,n} q_{o,n+1}  \right| \le \sum_{n=0}^\infty \frac{\log(1/g)}{g^\nu q_{o,n}^\nu} \le \frac{ \log(1/g)}{g^\nu}c_{1,\nu},$$
$$\left| \sum_{n=0}^\infty \beta_{o,n-1}^\nu \log \beta_{o,n-1} \right|\le \sum_{n=0}^\infty \frac{\log{(1/g)}+\log{q_{o,n}}}{ {g^\nu}q_{o,n}^\nu} \le \frac{c_{1,\nu}\log (1/g)+c_{2,\nu}}{g^\nu}.$$
From \eqref{eq:qx-p}, we have $q_{o,n}\beta_{o,n-1} + \veps_{o,n} q_{o,n-1} \beta_{o,n} = 1$ for all $n\ge 0$.
Thus
\begin{align}
&\nonumber \left| \sum_{n=0}^\infty \left(\frac{1}{(q_n^o)^\nu} - \beta_{o,n-1}^\nu\right) \log q_{o,n+1}\right|  \\
&\label{eq:logq/qO} \lesssim_\nu \sum_{n=0}^{\infty}{\max\left\{\frac{1}{q_{o,n}^{\nu-1}}, \beta_{o,n-1}^{\nu-1}\right\}}\left|\veps_{o,n} \frac{q_{o,n-1}}{q_{o,n}} \beta_{o,n} \right|\log q_{o,n+1}.
\end{align}
If $\nu\ge 1$, then \eqref{eq:logq/qO} is bounded above by 
$$\sum_{n=0}^\infty \beta_{o,n}\log q_{o,n+1} \le \frac{1}{g}\sum_{j=0}^\infty \dfrac{\log q_{o,n+1}}{q_{o,n+1}} \le \frac{c_{ {2,1}}}{g}.$$
If $0<\nu<1$, then 
\eqref{eq:logq/qO} bounded above by
\begin{align*}
& \sum_{n=0}^\infty \max\left\{\frac{1}{\beta_{o,n-1}^{1-\nu}}, q_{o,n}^{1-\nu}\right\} \dfrac{q_{o,n-1}}{q_{o,n}}\beta_{o,n}\log q_{o,n+1} 
 \le \sum_{n=0}^\infty \max \left\{\frac{\beta_{o,n}}{\beta_{o,n-1}^{1-\nu}},\frac{q_{o,n-1}}{gq_{o,n}^\nu q_{o,n+1}}\right\}\log q_{o,n+1} \\
& \le \sum_{n=0}^\infty \max \left\{\beta_{o,n}^{\nu},\frac{1}{g (q_{o,n+1})^\nu}\right\}\log q_{o,n+1} 
\le  \sum_{n=0}^\infty \frac{1}{g q_{o,n+1}^\nu}\log q_{o,n+1} \le \frac{c_{1,\nu}}{g}.
\end{align*}

\end{proof}

\subsection{H\"{o}lder continuity}
In this subsection, we prove that $B_{ {1,\nu}}-B_{odd, \nu}$ is $\nu/2$-H\"{o}lder continuous for  {$0<\nu\le 2$}.
In order to prove it, we investigate an operator which is derived from the functional equation of $B_{odd, \nu}$ as in \eqref{eq:fun.o}.
Let us denote by $X$ the space of measurable functions $f$ satisfies \eqref{eq:even2Z}.
Let $T_{odd,\nu}f$ be a function in $X$ whose values on $(0,1)$ are defined by
\begin{equation}\label{eq:TO}
T_{odd,\nu} f(x) = x^\nu f(1-x^{-1}) \text{ for } x\in(0,1).
\end{equation}
For $p\in [1,\infty]$, we consider the space 
$$X_{odd,p} := \{f \text{ satisfies \eqref{eq:even2Z} and } f|_{[0,1]}\in L^p([0,1],m_{odd})\}$$
equipped with the $p$-norm $\|f \|_p =\left(\int_0^1 |f|^p dm_{odd}\right)^{1/p}$.
The OCF-Brjuno function $B_{odd,\nu}$ is a function in an even $2\mathbb{Z}$-periodic function satisfying 
$$[(1-T_{odd,\nu})B_{odd,\nu}](x) = - \log(\|x\|_{2\mathbb{Z}}).$$
Since $m_{odd}$ is an absolutely continuous measure and its density function is bounded above and below, as {\it vector spaces} $L^p([0,1],dx)$ and $L^p([0,1],dm_{odd})$ coincide, and we ensure that $-\log (\|x\|_{2\mathbb{Z}})\in X_{odd,p}$ for all $p\ge 1$.
The following theorem tells us that $1-T_{odd,\nu}$ is invertible and $B_{odd,\nu}$ is a unique solution of the above functional equation on the space $X_{odd,p}$ for every $p\ge 1$.

\begin{prop} \label{pr:spec.O}
For all $\nu\ge 0$ and all $p\ge 1$, the operator $T_{odd,\nu}$ is a bounded linear operator on $X_{odd,p}$, and 
$\rho(T_{odd,\nu})\le g^\nu,$
where $\rho(T)$ is the spectral radius of an operator $T$.
\end{prop}

\begin{proof}
For given $p\ge 1$ and $f\in X_{odd,p}$, we have
\begin{equation}\label{eq:it.TO}
T_{odd,\nu}^{n}f(x) = \beta_{o,n-1}^{\nu}(x)f(A_{odd}^n(x)).
\end{equation}
By \eqref{eq:beO}, we have
$$\int_0^1|T_{odd,\nu}^n f(x)|^p \mathrm{d}m_{odd}(x) = \int_0^1\beta_{o,n-1}^{\nu p}(x)|f(A_{odd}^{n}(x))|^p \mathrm{d}m_{odd}(x)\le g^{(n-1) \nu p}\int_0^1|f|^p \mathrm{d}m_{odd}.$$
Then, the operator norm $\|T_{odd,\nu}^n \|_p\le g^{(n-1)\nu}$, which implies $\rho(T_{odd,\nu})\le g^{\nu}$.
\end{proof}

We call $I_n = [\frac{1}{n+1},\frac{1}{n}]$ for $n\in \mathbb{N}$ the \emph{branches of $A_{odd}$} on which the restriction of $A_{odd}$ is invertible.
\emph{The branches of $A_{odd}^n$} are defined by
$$I_{m_1} \cap A_{odd}^{-1} I_{m_2} \cap A_{odd}^{-2} I_{m_3} \cdots \cap A_{odd}^{-{(n-1)}} I_{m_n},$$
where $\{m_i\}_1^n \in \mathbb{N}^n$.

The \emph{splitting number $n(x,x')$} for $x,~x'\in[0,1]$  {is} the greatest integer $n$ such that $x$ and $x'$ belong to the same branch or to two adjacent branches of $A_{odd}^n$.
We denotes by $\delta(x,x') := n(x,x')-m$, where $m$ is the greatest integer such that $x$, $x'$ belong to the same branch of $A_{odd}^{m}$.
From now on, for $x,~x'\in [0,1]$, let us denote by 
\begin{equation*}\label{eq:drop.o}\begin{cases}
x_n =  {A_{odd}^n(x)}, ~ a_n = a_{o,n}(x), ~ \veps_n = \veps_{o,n}(x), ~ p_n = p_{o,n}(x), ~ q_n = q_{o,n}(x), \\
x_n' =  {A_{odd}^n(x')}, ~ a_n' = a_{o,n}(x'), ~ \veps_n' = \veps_{o,n}(x'), ~ p_n' = p_{o,n}(x'), ~ q_n' = q_{o,n}(x').
\end{cases}\end{equation*}
\begin{rmk}\label{rk:sp.n.O}
We write $n = n(x,x')$ and $\delta = \delta(x,x')$.
For all $\ell \le n-\delta$, we have
\begin{align}
\label{eq:l<n-del1}& a_\ell = a_\ell', ~ \veps_\ell = \veps_\ell', ~ p_\ell = p_\ell', ~ q_\ell = q_\ell', \\
\label{eq:l<n-del2}& |\beta_{o,\ell}(x) - \beta_{o,\ell}(x')| = |(q_\ell x - p_\ell) - (q_\ell' x' - p_\ell')| = q_\ell |x-x'| ~\text{(by \eqref{eq:qx-p})},\\
\label{eq:l<n-del3}& |x-x'| = |\beta_{o,\ell}(x)\beta_{o,\ell-1}(x') - \beta_{o,\ell-1}(x)\beta_{o,\ell}(x')| = | x_\ell -x_\ell ' |\beta_{o,\ell-1}(x)\beta_{o,\ell-1}(x').
\end{align}
Dividing by the following four cases, we will see that $\delta=0,1,2$ and a relation between $\{a_n,\veps_n, p_n,q_n\}$ and $\{a_n',\veps_n',p_n',q_n'\}$.
\begin{enumerate}[label=(\Alph*)]
\item \label{it:A} If the points $x,x'$ belong to the same branch of $A_{odd}^n$, then $\delta=0$.
\item \label{it:B} If the points $x,x'$ belong to the same branch of $A_{odd}^{n-1}$, and there is $k\ge1$ such that
\begin{equation}\label{eq:it:B}
\frac{1}{2k+2}\le x_{n-1} \le \frac{1}{2k+1} \le x_{n-1}' \le \frac{1}{2k},
\end{equation}
then $\delta=1$.
We have
$$a_n=a_n' = 2k+1, ~ p_n = p_n', ~ q_n = q_n', ~ \veps_n = 1, ~\veps_n'=-1.$$
Let $x'' : = p_n/q_n$  {and $x_n'': = A_{odd}^n(x'')$}.  {Then $x''$ is between $x$ and $x'$.}
We have $x_{n-1}'' = (2k+1)^{-1}$, $\beta_{o,n-1}(x'')=q_n^{-1}$, $x_n'' = 0$.
 {By \eqref{eq:det.pq},} we have $ |x-x'' | =  |(q_nx-p_n)(q_{n-1}x''-p_{n-1})-(q_{n-1}x-p_{n-1})(q_nx''-p_n) |$. 
 Combining with \eqref{eq:qx-p}, we have 
\begin{equation}\label{eq:B}
|x-x''| = q_n^{-1}\beta_{o,n}(x) \quad \text{and} \quad |x'-x''| = q_n^{-1}\beta_{o,n}(x').
\end{equation}

\item \label{it:C} If the points $x,x'$ belong to the same branch of $A_{odd}^{n-2}$ and there is $k \ge 1$ such that
\begin{equation}\label{eq:it:C}
\frac{2}{4k+1} \le x_{n-2} \le \frac{1}{2k} \le x_{n-2}' \le \frac{2}{4k-1},
\end{equation}
then $\delta=2$, and $x_{n-1}, x_{n-1}' \in [1/2,1]$.
We have
$$a_{n-1} = a_{n-1}' + 2 = 2k+1,~ \veps_{n-1} = - \veps_{n-1}' = -1, ~ a_n = a_n' = 1, ~ \veps_n = \veps_n' = 1, $$
\begin{equation*}
\begin{cases}
q_{n-1} = q_{n-1}' + 2q_{n-2}, ~ \\
p_{n-1} = p_{n-1}' + 2p_{n-2}, 
\end{cases}
\begin{cases}
q_n = q_n' = (q_{n-1}+q_{n-1}')/2, \\
p_n = p_n' = (p_{n-1}+p_{n-1}')/2.
\end{cases}
\end{equation*}
Let $x'' := p_n/q_n$. Then $x''$ is between $x$ and $x'$.
We have $x_{n-2}'' = (2k)^{-1}$, $x_{n-1}'' = 1$ and $x_n'' = 0$.
Then $\beta_{o,n-2}(x'') = \beta_{o,n-1}(x'') = q_n^{-1}$.
 {By the same argument in \ref{it:B}, we have}
\begin{equation}\label{eq:C}
|x-x''| = q_n^{-1}\beta_{o,n}(x)\quad \text{and} \quad |x'-x''| = q_n^{-1}\beta_{o,n}(x').
\end{equation}

\item \label{it:D} If the points $x, x'$ belong to the same branch of $A_{odd}^{n-1}$ and there is $k\ge 1$ such that
\begin{equation}\label{eq:it:D}
\frac{1}{2k+1}\le x_{n-1} \le \frac{1}{2k} \le x_{n-1}' \le \frac{1}{2k-1},
\end{equation}
but at least one of $x_n$, $x_n'$ is not in $[1/2,1]$, i.e.
\begin{equation}\label{eq:D:max}
\max\{1-x_n, 1-x_n'\}\ge 1/2,
\end{equation}
then $\delta=1$.
We have
\begin{equation}\label{eq:D}
\left\{ \begin{split}
& a_n = 2k+1, ~\veps_n = -1, \\
& a_n' = 2k-1, ~\veps_n = 1,
\end{split}\right. \qquad
\left\{\begin{split}
& q_n = q_n' + 2 q_{n-1}, \\
& p_n = p_n' + 2 p_{n-1}.
\end{split}\right.
\end{equation}
Let $x'' := \frac{p_n+p_n'}{q_n+q_n'}$, $a_i'' := a_{o,i}(x'')$ and $\veps_i'':= \veps_{o,i}(x'')$.
Then $x''$ is between $x$ and $x'$. 
We have $a_i'' = a_i$ and $\veps_i'' = \veps_i$ for $i\le n-1$.
Since $x_{n-1}''= (2k)^{-1}$, $x_n'' = 1$, from the same argument in \ref{it:B}, we have
\begin{equation}\label{eq:D:x-x''}
\begin{split}
 |x-x''| = \frac{2\beta_{o,n-1}(x)}{q_n+q_n'}(1-x_n) \quad \text{ and }\quad
 |x'-x''| = \frac{2\beta_{o,n-1}(x')}{q_n+q_n'}(1-x_n').
\end{split}
\end{equation}
\end{enumerate}
\end{rmk}

The following theorems extend the results showed in \cite{MMY97} for the NICF map $A_{1/2}$ to the OCF map $A_{odd}$.
The theorems tell us how the operator $T_{odd,\nu}$ preserves the H\"older continuity.
\begin{thm}\label{th:op1}
Let $f\in \mathcal{C}^\eta_{[0,1]}$ for $0<\eta\le 1$. Let $\overline{\eta} = \min\{\eta,\nu/2\}$.
\begin{enumerate}
\item\label{it:th:op1.1} For any $m\ge 0$ and $\nu>0$, $T_{odd,\nu}^m f$ is $\overline{\eta}$-H\"{o}lder continuous.
\item\label{it:th:op1.2} If $\nu>2\eta$ (thus $\overline{\eta}=\eta$), then $T_{odd,\nu}$ is a bounded linear operator on $\mathcal{C}^\eta_{[0,1]}$ with $\rho(T_{odd,\nu}|_{\mathcal{C}^\eta_{[0,1]}})\le g^{\nu-2\eta}<1$.
Therefore, the operator ${\bf{B}_{odd,\nu}} := (1-T_{odd,\nu})^{-1}$ is  {well-}defined in  {$\mathcal{C}^{\eta}_{[0,1]}$} and fulfills ${\bf{B}_{odd,\nu}} = \sum_{m=0}^\infty T_{odd,\nu}^m$.
\item\label{it:th:op1.3} If $\nu=2\eta$ (thus $\overline{\eta} = \eta = \nu/2$), then there exists $c_{3,\nu}>0$ such that
$$\|T_{odd,\nu}^m\|_{\mathcal{C}^{\nu/2}} \le c_{3,\nu}~ \text{ for all } m\ge 0.$$ 
\end{enumerate}
\end{thm}

\begin{thm}\label{th:op2}
Let $f\in \mathcal{C}^\eta_{[0,1]}$ for $0<\eta\le 1$.
\begin{enumerate}
\item \label{it:th:op2.1} If $\nu < 2\eta$, then the function ${\bf B_{odd,\nu}} f = \sum_{m\ge 0} T_{odd,\nu}^m f$ is in $\mathcal{C}^{\nu/2}_{[0,1]}$.
\item If $\nu = 2\eta$, then the function ${\bf B_{odd,\nu}} f$ is in $\mathcal{C}^{\eta'}_{ {[0,1]}}$ for $0 {<} \eta' < \nu/2$.
\end{enumerate}
\end{thm}

To prove the above theorem, we need to estimate f $|T_{odd,\nu}^mf(x) - T_{odd,\nu}^m f(x')|$ in terms of $|x-x'|^{\nu/2}$. For this aim, we state the following auxiliary properties. 

\begin{lem}\label{le:MMYlemO}
Let $x,~x'\in [0,1]$.
\begin{enumerate}
\item \label{it:lemO1} There  {exists} $\kappa_1 >1$ such that for all $\ell<n(x,x')$, 
\begin{equation}\label{eq:it:lemO1}
\kappa_1^{-1}\beta_{o,\ell}(x')<\beta_{o,\ell}(x)<\kappa_1 \beta_{o,\ell}(x').
\end{equation}

\item\label{it:lemO2} There  {exists} $\kappa_2> 1$ such that for all $\ell\ge n(x,x')$, 
\begin{equation}\label{eq:it:lemO2}
\max\{\beta_{o,\ell}(x),\beta_{o,\ell}(x')\}\le \kappa_2|x-x'|^{1/2}.
\end{equation}

\item \label{it:lemO3} Let $J$ be a branch of $A_{odd}^m$ which has an end point on $p_{o,m}/q_{o,m}$.
Then we have
$$\frac{1-g}{q_{o,m}^2}\le |J|\le \frac{g^{-1}}{q_{o,m}^2}, \quad \text{and}\quad  g^2 q_{o,m}^2 \le\left|\frac{dA_{odd}^m(x)}{dx}\right|\le \frac{q_{o,m}^2}{(1-g)^2}~\text{ for }x\in J.$$

\item \label{it:lemO4} Let $J,~J'$ be adjacent branches of $A_{odd}^m$.
Then there exists $\kappa_3 >1$ such that $$\kappa_3^{-1}|J'|\le |J| \le \kappa_3|J'|,$$
where $|J|$ is the length of $J$.
\end{enumerate}
\end{lem}

Since the proof are analogous to \cite[Lemma 1.10-13, Theorem 4.2, 4.4]{MMY97}, we give a full proof of the above theorems and lemma in Appendix~\ref{ap:B}.

\begin{prop}\label{pr:BO+-}
Let 
$$B_{odd,\nu}^+(x) = \frac{B_{odd,\nu}(x)+B_{odd,\nu}(1-x)}{2}\text{ and }B_{odd,\nu}^-(x) = \frac{B_{odd,\nu}(x)-B_{odd,\nu}(1-x)}{2}.$$
Then, $B_{odd,\nu}^+(x)$ is an even $\mathbb{Z}$-periodic function such that
\begin{equation*}\label{eq:BO+} 
B^+_{odd,\nu}(x) 
=  \log\frac{1}{x}- \frac{1-(1-x)^\nu}{2} \log \frac {1-x}{x} +x^\nu B_{odd,\nu}^+\left(\frac{1}{x}\right) \text{ for }x\in(0,1/2),
\end{equation*}
and $B_{odd,\nu}^-$ is an even $2\mathbb{Z}$-periodic function such that $B_{odd,\nu}^-(1+x)= - B_{odd,\nu}^-(x)$ and
\begin{equation}\label{eq:BO-} 
B^{-}_{odd,\nu}(x) 
=  \frac{1-(1-x)^\nu}{2}\log\frac {1-x}{x} + x^\nu B_{odd,\nu}^-\bigg(1-\frac{1}{x}\bigg) \text{ for }x\in(0,1/2).
\end{equation}
\end{prop}

\begin{proof}
By using \eqref{eq:fun.o}, for $x\in (0,1/2)$, we have
\begin{equation*}\begin{split}
B_{odd,\nu}(1-x) 
& = \log \frac{1}{1-x} + (1-x)^\nu B_{odd,\nu}\bigg(\frac{1}{1-x}-1\bigg) \\
& = \log \frac{1}{1-x} + (1-x)^\nu \bigg[\log \frac{1-x}{x} + \frac{x^\nu}{(1-x)^\nu} B_{odd,\nu} \bigg(\frac{1}{x}\bigg)\bigg] \\
& = \log \frac{1}{1-x} + (1-x)^\nu \log \frac{1-x}{x} + x^\nu B_{odd,\nu} \bigg(\frac{1}{x}\bigg).
\end{split}\end{equation*}
Combining \eqref{eq:fun.o} with the above equation, we have the conclusion.
\end{proof}

From Theorem~\ref{MMY4.6}, to show the H\"older continuity of $B_{1,\nu}-B_{odd,\nu}$, we will show that $B_{odd,\nu}-B_{1/2,\nu}$ is H\"older continuous.
We split $B_{odd,\nu}-B_{1/2,\nu}$ by an even $\mathbb{Z}$-periodic function $B_{odd,\nu}^+-B_{1/2,\nu}$ and an even $2\mathbb{Z}$-periodic function $B_{odd,\nu}^-$.
Thus it is enough to show that $B_{odd,\nu}^+-B_{1/2,\nu}\in \mathcal{C}^{\nu/2}_{[0,1/2]}$ and $B_{odd,\nu}^-\in\mathcal{C}^{\nu/2}_{[0,1]}$.

\begin{proof}[Proof of Theorem~\ref{main2}]
Let $\Delta := B^+_{odd,\nu} - B_{1/2,\nu}$.  {For $x\in(0,1/2]$,} we have
$$ 
\Delta(x)  =  -\frac{1 - (1-x)^\nu}{2} \log  {\frac{1-x}{x}} + x^\nu \left[B_{odd,\nu}^+\left(1-\frac{1}{x}\right)-B_{1/2,\nu}\left(1-\frac1x\right)\right].$$ 
Thus, $(1-T_{1/2,\nu})\Delta(x) =   - \frac{1 - (1-x)^\nu}{2} \log  {\frac{1-x}{x}}$,  {which} is in $\mathcal{C}^\eta_{[0,1/2]}$ for $\nu/2<\eta<\nu$. By Proposition~\ref{MMY4.4}, $\Delta$ is in $\mathcal{C}_{[0,1/2]}^{\nu/2}$.
Equation~\eqref{eq:BO-} is equivalent to
$$(1-T_{odd,\nu})B_{odd,\nu}^-(x) = \frac{1-(1-x)^\nu}{2}\log  {\frac{1-x}{x}} \in \mathcal{C}^{\eta}_{[0, 1]}\text{ for }\nu/2<\eta<\nu.$$
By Theorem~\ref{th:op2}-\eqref{it:th:op2.1}, we have $B_{odd,\nu}^- \in \mathcal{C}_{[0, 1]}^{\nu/2}$.
\end{proof}

\section{The ECF- and OOCF-Brjuno functions}\label{sec:EF}
In this section, we recall some basic properties of the ECF and the OOCF and define a Brjuno-like function associated with the OOCF.
We will prove Theorem~\ref{main3} which tells us that the corresponding Brjuno numbers can be characterized by a combination of the Brjuno functions associated  {with} the ECF and  {with} the OOCF.

\subsection{The even continued fraction}
\begin{figure} 
\begin{subfigure}{0.49\textwidth}
\centering
\begin{tikzpicture}[scale=5]
\draw (1.4,0) -- (1.4+1,0);
\draw (1.4+0,0) -- (1.4+0,1);
\draw[domain=1.4+1/2:1.4+1,black] plot (\x,{-(1/(\x-1.4)-2)});
\draw[domain=1.4+1/3:1.4+1/2,black] plot (\x,{(1/(\x-1.4)-2)});
\draw[domain=1.4+1/4:1.4+1/3,black] plot (\x,{-(1/(\x-1.4)-4)});
\draw[domain=1.4+1/5:1.4+1/4,black] plot (\x,{(1/(\x-1.4)-4)});
\draw[domain=1.4+1/6:1.4+1/5,black] plot (\x,{-(1/(\x-1.4)-6)});
\draw[domain=1.4+1/7:1.4+1/6,black] plot (\x,{(1/(\x-1.4)-6)});
\draw[domain=1.4+1/8:1.4+1/7,black] plot (\x,{-(1/(\x-1.4)-8)});
\draw[domain=1.4+1/9:1.4+1/8,black] plot (\x,{(1/(\x-1.4)-8)});
\draw[domain=1.4+1/10:1.4+1/9,black] plot (\x,{-(1/(\x-1.4)-10)});
\draw[domain=1.4+1/11:1.4+1/10,black] plot (\x,{(1/(\x-1.4)-10)});
\draw[domain=1.4+1/12:1.4+1/11,black] plot (\x,{-(1/(\x-1.4)-12)});
\draw[domain=1.4+1/13:1.4+1/12,black] plot (\x,{(1/(\x-1.4)-12)});
\draw[domain=1.4+1/14:1.4+1/13,black] plot (\x,{-(1/(\x-1.4)-14)});
\draw[domain=1.4+1/15:1.4+1/14,black] plot (\x,{(1/(\x-1.4)-14)});

\draw[domain=1.4+1/16:1.4+1/15,black] plot (\x,{-(1/(\x-1.4)-16)});
\draw[domain=1.4+1/17:1.4+1/16,black] plot (\x,{(1/(\x-1.4)-16)});
\draw[domain=1.4+1/18:1.4+1/17,black] plot (\x,{-(1/(\x-1.4)-18)});
\draw[domain=1.4+1/19:1.4+1/18,black] plot (\x,{(1/(\x-1.4)-18)});
\draw[domain=1.4+1/20:1.4+1/19,black] plot (\x,{-(1/(\x-1.4)-20)});
\draw[domain=1.4+1/21:1.4+1/20,black] plot (\x,{(1/(\x-1.4)-20)});

\draw[dotted, thin] (1.4+1/2,0) -- (1.4+1/2,1);

\node at (1.4-0.05,1){\tiny$1$};
\node at (1.4+1,-0.08){\tiny$1$};
\node at (1.4-0.05,-0.08){\tiny$0$};
\node at (1.4+1/2,-0.1) {\small$\frac12$};

\draw[dotted, thin] (1.4+1/3,0) -- (1.4+1/3,1);

\node at (1.4+1/3,-0.1) {\small$\frac13$};
\node at (1.4+1/4,-0.1) {\small$\frac14$};
\node at (1.4+1/5,-0.1) {\small$\frac15$};

\draw[dotted, thin](1.4+1/4,0)--(1.4+1/4,1);
\draw[dotted, thin](1.4+1/5,0)--(1.4+1/5,1);

\end{tikzpicture}
\caption{The graph of $\AE$.}\label{fig:AE}
\end{subfigure}
\begin{subfigure}{0.49\textwidth}
\centering
\begin{tikzpicture}[scale = 5]
\draw (1.4,0) -- (1.4+1,0);
\draw (1.4+0,0) -- (1.4+0,1);
\draw[domain=1.4+0:1.4+1/3, black] plot (\x, {(\x-1.4)/(1-2*(\x-1.4))});
\draw[domain=1.4+1/3:1.4+1/2, black] plot (\x,{1/(\x-1.4)-2});
\draw[domain=1.4+1/2:1.4+3/5, black] plot (\x, {(2*(\x-1.4)-1)/(2-3*(\x-1.4))});
\draw[domain=1.4+3/5:1.4+2/3, black] plot (\x, {(2-3*(\x-1.4))/(2*(\x-1.4)-1)});
\draw[domain=1.4+2/3:1.4+5/7, black] plot (\x, {(3*(\x-1.4)-2)/(3-4*(\x-1.4))});
\draw[domain=1.4+5/7:1.4+3/4, black] plot (\x, {(3-4*(\x-1.4))/(3*(\x-1.4)-2)});
\draw[domain=1.4+3/4:1.4+7/9, black] plot (\x, {(4*(\x-1.4)-3)/(4-5*(\x-1.4))});
\draw[domain=1.4+7/9:1.4+4/5, black] plot (\x, {(4-5*(\x-1.4))/(4*(\x-1.4)-3)});
\draw[domain=1.4+4/5:1.4+9/11, black] plot (\x, {(5*(\x-1.4)-4)/(5-6*(\x-1.4))});
\draw[domain=1.4+9/11:1.4+5/6, black] plot (\x, {(5-6*(\x-1.4))/(5*(\x-1.4)-4)});
\draw[domain=1.4+5/6:1.4+11/13, black] plot (\x, {(6*(\x-1.4)-5)/(6-7*(\x-1.4))});
\draw[domain=1.4+11/13:1.4+6/7, black] plot (\x, {(6-7*(\x-1.4))/(6*(\x-1.4)-5)});
\draw[domain=1.4+6/7:1.4+13/15, black] plot (\x, {(7*(\x-1.4)-6)/(7-8*(\x-1.4))});
\draw[domain=1.4+13/15:1.4+7/8, black] plot (\x, {(7-8*(\x-1.4))/(7*(\x-1.4)-6)});
\draw[domain=1.4+7/8:1.4+15/17, black] plot (\x, {(8*(\x-1.4)-7)/(8-9*(\x-1.4))});
\draw[domain=1.4+15/17:1.4+8/9, black] plot (\x, {(8-9*(\x-1.4))/(8*(\x-1.4)-7)});

\draw[dotted, thin] (1.4+1,0)--(1.4+1,1);
\draw[dotted, thin] (1.4+1/2,0) -- (1.4+1/2,1);

\node at (1.4-0.05,1){\tiny$1$};
\node at (1.4+1,-0.08){\tiny$1$};
\node at (1.4-0.05,-0.08){\tiny$0$};
\node at (1.4+1/2,-0.1) {\small$\frac12$};

\draw[dotted, thin] (1.4+1/3,0) -- (1.4+1/3,1);
\draw[dotted, thin] (1.4+3/5,0) -- (1.4+3/5,1);
\draw[dotted, thin] (1.4+2/3,0) -- (1.4+2/3,1);
\node at (1.4+1/3,-0.1) {\small$\frac13$};
\node at (1.4+3/5,-0.1) {\small$\frac35$};
\node at (1.4+2/3,-0.1) {\small$\frac23$};
\node at (1.4+3/4,-0.1) {\small$\frac34$};

\end{tikzpicture}
\caption{The graphs of $A_{oo}$.}\label{fig:AOO}
\end{subfigure}
\end{figure}
The  {ECF} map $A_{even}$  {as in \eqref{eq:AOE}} is ergodic with respect to
$dm_{even}(x) := \frac{dx}{1-x^2},$
which is a $\sigma$-finite measure with infinite total mass \cite[Thm 2]{Sch82}.
We denote by $\lfloor t\rfloor_{even}:= 2\lfloor (t+1)/2\rfloor$ a nearest even integer of $t$.
For $x\in \mathbb{R}$, we denote the \emph{ECF partial quotients} $(a_{e,n},\veps_{e,n})$ by
\begin{equation}\label{eq:digitE}
a_{e,0}:=\lfloor x \rfloor_{even},~ \veps_{e,0}:=\mathrm{sign}(x-a_{e,0}),~
a_{e,n} := \left\lfloor \frac{1}{x_{e,n-1}}\right\rfloor_{even} \text{ and }\veps_{e,n} := \text{sign}\bigg(\dfrac{1}{x_{e,n-1}} - a_{e,n}\bigg),
\end{equation}
where $x_{e,n}$ is an $A_{even}$-orbit of $x$ defined as in \eqref{eq:set.x_n}.
The \emph{ECF expansion} of $x$ is 
\begin{equation*}
a_{e,0}+ \dfrac{\veps_{e,0}}{a_{e,1}+\dfrac{\veps_{e,1}} {\ddots+\dfrac{\ddots}{a_{e,n}+ \dfrac{\veps_{e,n}}{a_{e,n+1}+\ddots}}}}.
\end{equation*}
The \emph{$n$th principal convergent of the ECF}  {is} the truncated continued fraction expansion denoted by
$$
\frac{p_{e,n}}{q_{e,n}} := a_{e,0} + \dfrac{\veps_{e,0}}{a_{e,1}+\dfrac{\veps_{e,1}}{a_{e,2}+\ddots+\dfrac{\veps_{e,n-1}}{a_{e,n}}}}.
$$
Then, $p_{e,n}/q_{e,n}$ for all $n\ge 0$ hold the same properties in Remark~\ref{rk:p/qO}.
Note that $\{q_{e,n}\}$ is an increasing sequence.

\subsection{The odd-odd continued fraction}
Kim, Liao and the first author defined the following interval map 
$$
A_{oo} (x) = 
\begin{cases}
\frac{k x - (k-1) }{k - (k+1) x}, &  x \in\left[\frac{k-1}{k},\frac{2k -1}{2k+1}\right], \vspace{0.2cm} \\
\frac{k - (k+1) x}{k x - (k-1) }, &   x \in \left[\frac{2k -1}{2k+1},\frac{k}{k+1}\right],
\end{cases}
\quad \text{ for } k\ge1,
$$
and investigated the corresponding continued fraction, which is called \emph{the odd-odd continued fraction (OOCF)} \cite{KLL22}.
See Figure~\subref{fig:AOO} for the graph of $A_{oo}$.
The  {OOCF expansion} has the form
$$
1 - \dfrac{1}{a_{oo,1} + \dfrac{\veps_{oo,1}}{ 2 - \dfrac{1}{a_{oo,2} + \dfrac{\veps_{oo,2}}{2-\dfrac{1}{ \ddots -\dfrac{1}{a_{oo,n}+\dfrac{\veps_{oo,n}}{2-\ddots}}}}}}},
$$
where $a_{oo,n} \in \mathbb N$ and $\veps_{oo,n} \in \{1, -1\}$ for $a_{oo,n} \ge 2$ and $\veps_{oo,n} = 1$ for $a_{oo,n} = 1$.
For $n\ge 1$, the \emph{$n$th principal convergent} of the OOCF is defined by
$$
\frac{p_{oo,n}}{q_{oo,n}}:=
1 -\cfrac{1}{a_{oo,1} + \dfrac{\veps_{oo,1}}{ 2 - \dfrac{1}{\ddots \dfrac{\veps_{oo,n-1}}{ 2 - \dfrac{1}{a_{oo,n} + \dfrac{\veps_{oo,n}}{2}}}}}}.
$$
 {Let $p_{oo,-1}=-1$, $q_{oo,-1}=1$, $p_{oo,0} = 1$ and $q_{oo,0}=1$.}
We have the following recursive relations:
\begin{equation*}
\begin{cases}
p_{oo,n}=(2a_{oo,n}+\veps_{oo,n} {-1})p_{oo,n-1}+\veps_{oo,n-1}p_{oo,n-2}, \\
q_{oo,n}=(2a_{oo,n}+\veps_{oo,n} {-1})q_{oo,n-1}+\veps_{oo,n-1}q_{oo,n-2},
\end{cases}
\quad \text{ for }n\ge   1.
\end{equation*}
We  {denote} by $x_{oo,n}=A_{oo}^{n}(x)$.
We  {define a function $\iota:[0,1]\to[0,1]$} by $\iota(x) = \frac{1-x}{1+x}$.
Note that $\iota^2=\text{id}$ and the dynamical systems $(I, A_{even}, m_{even})$ and $(I,A_{oo},m_{oo})$ are conjugate to each other via $\iota$.
We have
\begin{equation}\label{eq:x_nOO}
x = \frac{p_{oo,n}+p_{oo,n-1}\veps_{oo,n}\iota(x_{oo,n})}{q_{oo,n}+q_{oo,n-1}\veps_{oo,n}\iota(x_{oo,n})},\quad \text{and} \quad \iota(x_{oo,n})=-\veps_{oo,n}\frac{q_{oo,n}x-p_{oo,n}}{q_{oo,n-1}x-p_{oo,n-1}}.
\end{equation}
By letting
$$\beta_{oo,n}(x) := (x_{oo,0}+1)\prod_{i=0}^n \iota(x_{oo,i}),~~\text{ and }~~ \beta_{oo,-1}(x):=1,$$
we define the \emph{odd-odd-Brjuno function} with a positive exponent $\nu$ by 
\begin{equation}\label{eq:BOO}
B_{oo,\nu}(x) = - \sum_{n=0}^\infty (\beta_{oo,n-1}(x_{oo,n}))^\nu\log{\iota(x_{oo,n})}.
\end{equation}
We note that $B_{oo,\nu}(x) = (\{x\}+1)^\nu B_{even,\nu}(\iota(\{x\})).$ 
See Figure~\ref{fig:EOO} for the graph of $B_{oo, 1}$.

From \eqref{eq:x_nOO} and the fact that $|p_{oo,n+1}q_{oo,n} - p_{oo,n}q_{oo,n+1}| =2$, we have
\begin{equation}\label{eq:beOO}
\beta_{oo,n}(x) = |q_{oo,n}x-p_{oo,n}| =  \frac{2}{q_{oo,n+1} + \varepsilon_{oo,n+1}q_{oo,n}\iota(x_{oo,n+1})}.
\end{equation}

\subsection{Proof of Theorem~\ref{main3}}
We call  {a rational} $p/q$  {in lowest terms} an \emph{$\infty$-rational} if $p$ and $q$ have different parities, and a \emph{$1$-rational} if $p$ and $q$ are both odd.
The ECF principal convergents  {$p_{e,n}/q_{e,n}$} are the best-approximating  {rationals} among  {the} $\infty$-rationals \cite{SW14}.
On the other hand, $p_{oo,n}/q_{oo,n}$ are the best-approximating  {rationals} among  {the} $1$-rationals \cite{KLL22}.
To show Theorem~\ref{main3}, we separate the series $\sum_{n=0}^\infty (\log Q_{n+1})/{Q_n^\nu}$ by the sum over  {$n$ such that $P_n/Q_n$ is an $\infty$-rational, and the sum over $n$ such that $P_n/Q_n$ is a $1$-rational.}

\begin{prop}
There exists $C_\nu>0$ such that, for all $x\in \mathbb{R}\setminus\mathbb{Q}$, 
$$\left|\sum_{k=1}^\infty \frac{\log q_{e,k+1}}{(q_{e,k})^\nu}- \sum_{\frac{P_n}{Q_n}: \infty\text{-rational}}\frac{\log Q_{n+1}}{Q_n^\nu}\right|<C_\nu.$$
\end{prop}

\begin{proof}
Kraaikamp and Lopes investigated a relation between the ECF and the RCF convergents.
We recall the relation as follows.
If $p_{e,k}/q_{e,k} = P_n/Q_n$ for some $n$, then either
$q_{e,k+1} = Q_{n+1}$ or $q_{e,k+1} = Q_{n+1}+Q_n$.
If $p_{e,k}/q_{e,k}$ is not an RCF convergent, then
$q_{e,k} = (i-1)Q_n+Q_{n-1}$ and $q_{e,k+1}=iQ_n+Q_{n-1}$ for some $2\le i \le a_{n+1}$, 
see \cite[Remark 2 and Proof of Lemma 1]{KL96} for more details.

Thus, if $p_{e,k}/q_{e,k}$ is a RCF convergent, then 
$$ {0\le} \frac{\log q_{e,k+1}}{(q_{e,k})^\nu} - \frac{\log Q_{n+1}}{Q_n^\nu} \le  \frac{\log 2}{Q_n^\nu}.$$
If $p_{e,k}/q_{e,k}$ is not a RCF convergent, then 
$$\frac{\log q_{e,k+1}}{(q_{e,k})^\nu} = \frac{\log (iQ_n+Q_{n-1})}{((i-1)Q_n+Q_{n-1})^\nu}\le \frac{\log Q_n+\log  {s_\nu}}{Q_n^\nu}.$$
 {where $s_\nu := \sup_{i\ge 2} \frac{\log(i+1)}{(i-1)^\nu}<\infty.$}
Since every $\infty$-rational RCF convergent is an ECF convergent, we have
$$\left|\sum_{k=1}^\infty \frac{\log q_{e,k+1}}{(q_{e,k})^\nu}- \sum_{\frac{P_n}{Q_n}:\infty\text{-rational}}\frac{\log Q_{n+1}}{Q_n^\nu}\right|
\le \sum_{n=1}^\infty \frac{\log Q_n+\log  {s_\nu}}{Q_n^\nu} \le C_{1,\nu} \log  {s_\nu}  + C_{2,\nu},
$$
here $C_{1,\nu}$, $C_{2,\nu}$ are constants as in \eqref{eq:sum.1/qR}.
\end{proof}

\begin{rmk}\label{rk:OO}
Let $x = [a_0;a_1,a_2,\cdots,a_n,\cdots]$.
The \emph{intermediate convergent} $P_{n, i}/Q_{n,i}$ is defined by 
$$P_{n,i}=iP_{n-1}+P_{n-2}\quad\text{and}\quad Q_{n,i}=iQ_{n-1}+Q_{n-2},\quad \text{ for }0\le i \le a_n.$$

Every $p^{oo}_k/q^{oo}_k$ is a RCF principal convergent or an intermediate convergent.
If $P_{n}/Q_{n}$ is a $1$-rational, then $P_{n+1,i}/Q_{n+1,i}$ is not an OOCF principal convergent for any $1\le i < a_n$.
If $P_{n}/Q_{n}$ is an $\infty$-rational, then every $1$-rational $P_{n+1,i}/Q_{n+1,i}$ is an OOCF principal convergent, see \cite[Section 5]{KLL22}.
\end{rmk}

\begin{prop}
There exists $C_\nu>0$ such that, for all $x\in \mathbb{R}\setminus\mathbb{Q}$, one has
$$\left|\sum_{k=1}^\infty \frac{\log q_{oo,k+1}}{q_{oo,k}}- \sum_{\frac{P_n}{Q_n}:1\text{-rational}}\frac{\log Q_{n+1}}{Q_n}\right|<C_\nu.$$
\end{prop}

\begin{proof}
Since if $P_n/Q_n$ is a $1$-rational, then $P_{n+1}/Q_{n+1}$ is an $\infty$-rational, from the facts in Remark~\ref{rk:OO}, 
if $p_{oo,k}/q_{oo,k}$ is a RCF principal convergent $P_n/Q_n$, 
then $q_{oo,k+1} = 2Q_{n+1}+Q_n$.
Then, 
$$\frac{\log q_{oo,k+1}}{(q_{oo,k})^\nu} - \frac{\log Q_{n+1}}{Q_n^\nu} \le \frac{\log 3}{Q_n^\nu}.$$
Moreover, if $p_{oo,k}/q_{oo,k}$ is not a RCF principal convergent, then we have $q_{oo,k} = iQ_n + Q_{n-1}$ and ${q_{oo,k+1}} = {(i+2)Q_{n}+Q_{n-1}}$ for some $n$ and $1\le i\le a_n-2$.
Then, 
$$\frac{\log q_{oo,k+1}}{(q_{oo,k})^\nu}  = \frac{\log((i+2)Q_n+Q_{n-1})}{(iQ_n+Q_{n-1})^\nu} \le \frac{\log Q_n}{Q_n^\nu} + \frac{\log(i+3)}{(iQ_n)^\nu} \le \frac{\log Q_n + \log  {s_\nu}}{Q_n^\nu},$$
 {where $s_\nu := \sup_{i\ge 1} \frac{\log(i+3)}{i^\nu}<\infty.$}
 {In summary,} we have
$$\left|\sum_{k=1}^\infty \frac{\log q_{oo,k+1}}{(q_{oo,k})^\nu}- \sum_{\frac{P_n}{Q_n}:1\text{-rational}}\frac{\log Q_{n+1}}{Q_n^\nu}\right|
< \sum_{n=1}^\infty \frac{\log  {3s_\nu} + \log Q_n}{Q_n^\nu}\le C_{1,\nu}\log {(3s_\nu)}  + C_{2,\nu},$$
 {where $C_{1,\nu}$ and $C_{2,\nu}$ are constants in \eqref{eq:sum.1/qR}.}
\end{proof}

\begin{prop}\label{pr:est.BE}
There exists $C_\nu>0$ such that, for all $x\in\mathbb{R}\setminus\mathbb{Q}$, one has
$$\left| B_{even,\nu}(x)-\sum_{n=0}^\infty \frac{\log q_{e,n+1}}{(q_{e,n})^\nu}\right|\le C_\nu.$$
\end{prop}

\begin{prop}\label{pr:est.BOO}
There exists $C_\nu>0$ such that, for all $x\in\mathbb{R}\setminus\mathbb{Q}$, one has
$$\left|\frac 12 B_{oo,\nu}(x)-\sum_{n=0}^\infty\frac{\log q_{oo,n+1}}{(q_{oo,n})^\nu}\right|\le C_\nu.$$
\end{prop}

We can show the above propositions in a similar way to \cite[Theorem 13]{LMNN10}.
For completeness, we give their proofs in Appendix~\ref{ap:C}.
Combining the above propositions with \eqref{eq:B-logq/q}, we have Theorem~\ref{main3}.

\begin{rmk}
The operator $T_{1,\nu}$ in \eqref{eq:op.alpha} associated to the classical Brjuno function has the same spectral property to Proposition~\ref{pr:spec.O}, see \cite[Theorem 2.6]{MMY01}, which implies $B_{1,\nu}\in L^p(dx)$.
Since $B_{even,\nu}$ and $B_{oo,\nu}$ are positive functions, $\|B_{even,\nu}\|_p \le \|B_{even,\nu}+\frac{1}{2}B_{oo,\nu}\|_p = \|B_{1,\nu}+\varphi\|_p<\infty$ with an $L^\infty$ function $\varphi$.

However, $L^p(dx)$ is not equivalent to $L^p(dm_{even})$. The space $L^p(dm_{even})$ is a proper subspace of $L^p(dx)$.
It is caused from the infinite total mass of $dm_{even} = \frac{dx}{1-x^2}$.
The operator associated to the ECF-Brjuno function on $L^p(dm_{even})$ has spectral radius exactly $1$ which does not guarantee the existence of the solution of the functional equation \eqref{eq:fun.e} in $L^p(dm_{even})$.

If $\nu\le 1$, then $B_{even,\nu}$ is not $L^p(dm_{even})$ for all $p\ge 1$.
If $\nu=1$, then the value $B_{even,\nu}$ in the neighborhood of $1$ has positive constant lower bound, and if $\nu<1$, then $B_{even,\nu}(x)$ goes to infinity when $x$ goes to $1$. If $\veps \ll \frac{1}{2}$, then $A_{even}^{k}(x)\in [\frac12,1]$ for $i\le \lceil \frac{1}{\veps}-2\rceil$.
Then, 
$$B_{even,\nu}(1-\veps) = \sum_{k=0}^n (1-k\veps)^\nu \log \frac{1-k\veps}{1-(k+1)\veps} + (1-(n+1)\veps)^\nu B_{even,\nu}\left(\frac{1-n\veps}{1-(n+1)\veps}\right),$$ where
$n= \lceil\frac{1}{\veps}-2\rceil$.
Since $-\log x \gtrsim 1-x$ near $1$, each term of the second summand $(1-k\veps)^\nu \log \frac{1-k\veps}{1-(k+1)\veps}\gtrsim \veps(1-k\veps)^{\nu-1}$. Combining with $(1-k\veps)<2\veps$ and $\nu\le 1$, we have $\sum_{k=0}^n(1-k\veps)^{\nu-1}\ge \frac{\frac1{\veps}-1}{(2\veps)^{1-\nu}}$.
Thus we have
$$\int_0^1 \frac{B_{even,\nu}(x)^p dx}{1-x^2}\ge \int_0^{1/2} \frac{(B_{even,\nu}(1-\veps))^p d\veps}{\veps(2-\veps)} \ge \int_0^{1/2}\left(\frac{1-\veps}{\veps^{1-\nu}}\right)^p\frac{d\veps}{\veps}\gtrsim \int_0^{1/2}\frac{d\veps}{\veps^{(1-\nu)p+1}}=\infty.$$
\end{rmk}

\appendix

\section{Proof of Lemma~\ref{le:MMYlemO}, Theorem~\ref{th:op1} and \ref{th:op2}}\label{ap:B}

\begin{proof}[Proof of Lemma~\ref{le:MMYlemO}]
Recall that we denote by $n=n(x,x')$, $\delta=\delta(x,x')$  and $g=\frac{\sqrt{5}-1}{2}$.
We follow the notations in \eqref{eq:drop.o}. 

(1) For $\ell<n-\delta$, by \eqref{eq:beq} and \eqref{eq:l<n-del1},
we have $g(1-g)\le \frac{\beta_{o,\ell}(x)}{\beta_{o,\ell}(x')} \le g^{-1}(1-g)^{-1}.$
For $n-\delta\le i \le n-1$, by \eqref{eq:it:B}, \eqref{eq:it:C} and \eqref{eq:it:D}, we have $\frac{1}{3} \le x_i/x_i' \le 3$.
Thus, we have the conclusion by taking $\kappa_1 =  {9g^{-1}(1-g)^{-1}}$.

(2) Since $\beta_{o,\ell} \le \beta_{o,n}$ for $\ell\ge n$, it is enough to show that \eqref{eq:it:lemO2}  {holds} for $\ell=n$.
For each case \ref{it:B}, \ref{it:C} and \ref{it:D}, we have $|x-x'| = |x-x''|+|x''-x|$.

In the case \ref{it:D},  without loss of generality, let us assume that $q_n>q_n'$.
From \eqref{eq:D:x-x''}, \eqref{eq:D:max} and \eqref{eq:beq}, we have
$$|x-x'| \ge \frac{\beta_{o,n-1}(x)+\beta_{o,n-1}(x')}{q_n+q_n'} \ge \frac{1-g}{2q_n^2}\ge \frac{(1-g)g^2}{2}\beta_{o,n-1}^2(x).$$
By Lemma~\ref{le:MMYlemO}-\eqref{it:lemO1}, we have $|x-x'|\ge \frac{(1-g)g^2}{2}\kappa_1^{-2}\beta_{o,n-1}^2(x')$, then 
\begin{equation}\label{eq:le3.1}
\max\{\beta_{o,n}(x),\beta_{o,n}(x')\}\le \sqrt{2} \kappa_1 g^{-1}(1-g)^{-1/2}|x-x'|^{1/2}.
\end{equation}

In the case \ref{it:B} and \ref{it:C}, from \eqref{eq:B} and \eqref{eq:C}, we have
$$|x-x'|=q_n^{-1}(\beta_{o,n}(x)+\beta_{o,n}(x'))\ge q_n^{-1}\max\{\beta_{o,n}(x),\beta_{o,n}(x')\}.$$
From \eqref{eq:beq} and Lemma~\ref{le:MMYlemO}-\eqref{it:lemO1}, we have $q_n^{-1}\ge g \beta_{o,n-1}(x)\ge \kappa_1^{-1} g \beta_{o,n-1}(x')$.
Since $\beta_{o,n-1}(x)\ge \beta_{o,n}(x)$, we have $|x-x'|\ge g \max\{\kappa_1^{-2} \beta_{o,n}^2(x), \kappa_1^{-1}\beta_{o,n}^2(x')\}$, which implies that
\begin{equation}\label{eq:le3.2}
\max\{\beta_{o,n}(x),\beta_{o,n}(x')\}\le g^{-1/2}\kappa_1|x-x'|^{1/2}.
\end{equation}

In the case \ref{it:A}, 
we have $|x_n^{-1}-{x'}_n^{-1}|\ge 1$
 since $x$ and $x'$ do not belong to two adjacent branches of 
$A_{odd}^{n+1}$. Thus we have $|x_n-x_n'|\ge x_n x_n'$.
Suppose that $x_n>x_n'$.
If $x_n'\ge x_n/2$, then we have $|x_n-x_n'|\ge x_n^2/2$.
If $x_n' < x_n/2$, then we also have $|x_n-x_n'|=x_n\left|1-\frac{x_n'}{x_n}\right|\ge x_n/2 >x_n^2/2$.
Symmetrically, we can show that if $x_n\le x_n'$, then we have $|x_n-x_n'| \ge {x'}_n^2/2$.
From \eqref{eq:l<n-del3}  {and Lemma~\ref{le:MMYlemO}-\eqref{it:lemO1}}, we have
\begin{align*}
|x-x'| 
& \ge \frac{1}{2}\beta_{o,n-1}(x)\beta_{o,n-1}(x')\max\{x_n^2,{x'}_n^2\} \\
& \ge \frac{\max\{\beta_{o,n-1}(x),\beta_{o,n-1}(x')\}^2\max\{x_n^2,{x'}_n^2\}}{2\kappa_1^{  2}} \ge \frac{1}{2\kappa_1}\max\{\beta_{o,n}(x),\beta_{o,n}(x')\}^2,
\end{align*}
which implies that
\begin{equation}\label{eq:le3.3}
\max\{\beta_{o,n}(x),\beta_{o,n}(x')\}\le \sqrt{2\kappa_1}|x-x'|^{1/2}.
\end{equation}

Combining \eqref{eq:le3.1}, \eqref{eq:le3.2} and \eqref{eq:le3.3}, we have $\max\{\beta_{o,n}(x),\beta_{o,n}(x)\}\le \kappa_2 |x-x'|^{1/2}$ by letting $\kappa_2 =\sqrt{2} \kappa_1 g^{-1}(1-g)^{-1/2}$.

(3)
By \eqref{eq:x_n}, the end-points of $J$ are attained by putting $x_m=0$ or $1$ in $\frac{p_m+p_{m-1}\veps_{m}x_m}{q_m+q_{m-1}\veps_{m}x_m}$, which are $\frac{p_m}{q_m}$ and $\frac{p_m + \veps_m p_{m-1}}{q_m+ \veps_{m} q_{m-1}}$.
Thus $|J|^{-1} = \left|\frac{p_m}{q_m}-\frac{p_m + \veps_{m} p_{m-1}}{q_m+ \veps_{m} q_{m-1}}\right|^{-1} = q_m(q_m + \veps_{m} q_{m-1})$.
By \eqref{eq:x_n}, \eqref{eq:det.pq} and \eqref{eq:qx-p}, we have $\left|\frac{dA_{odd}^m(x)}{dx}\right| = |q_m+q_{m-1}\veps_{m}x_m|^2$ for $x\in J$.
From \eqref{eq:q.O3},  we have $q_{m}^2(g^{-1}-1)<  |J|^{-1} < q_m^2(g^{-1}+1)$  {and} $q_m^2(g^{-1}-1)^2 < \left|\frac{dA_{odd}^m(x)}{dx}\right|  < q_m^2(g^{-1}+1)^2$.

(4) 
Let $x$ be  {the} common end-point of $J$ and $J'$.

If $A_{odd}^m(x)=0$, then we have $x= p_m/q_m$ and the other end-points are $\frac{p_m\pm p_{m-1}}{q_m \pm q_{m-1}}$ by \eqref{eq:x_n}.
By Lemma~\ref{le:MMYlemO}-\eqref{it:lemO3}, $ {\frac{g^{-1}-1}{g^{-1}+1}}\le |J|/|J'|\le  {\frac{g^{-1}+1}{g^{-1}-1}} = 2g^{-1}+1$.

If $A_{odd}^{m}(x)=1$, then we have $x= \frac{p_m+\veps_m p_{m-1}}{q_m+\veps_m q_{m-1}}$ and one of other end-points is $\frac{p_m}{q_m}$.
Let $\frac{p_m'}{q_m'}$ be the other end-point.
By \eqref{eq:D}, we have $p_m = p_m' + 2 p_{m-1}$ and $q_m = q_m' + 2 q_{m-1}$.
Since 
$\left|\frac{p_m+\veps_m p_{m-1}}{q_m+\veps_m q_{m-1}}-\frac{p_m}{q_m}\right| 
= \frac{1}{q_m|q_m+\veps_mq_{m-1}|}$ and 
$\left|\frac{p_m+\veps_m p_{m-1}}{q_m+\veps_m q_{m-1}}-\frac{p_m'}{q_m'}\right|
 = \frac{2+\veps_m}{q_m'|q_m+\veps_{m} q_{m-1}|}$,
we have $|J|/|J'| =  \frac{q_m'}{q_m(2+\veps_m)}$ or $\frac{(2+\veps_m)q_m}{q_m'}$.
By \eqref{eq:l<n-del1}, we have $p_{m-1}=p_{m-1}'$ and $q_{m-1}=q_{m-1}'$.
By \eqref{eq:q.O2}, $q_m/q_m' = 1+2q_{m-1}'/q_m'\le 1+ 2g^{-1}$.
Thus $1\le \frac{(2+\veps_m)q_m}{q_m'}\le 3(1+2g^{-1})$.

We have a conclusion by letting $\kappa_3 = 3(1+2g^{-1})$.
\end{proof}

\begin{proof}[Proof of Theorem~\ref{th:op1}]
(1) Let $x,~x' \in [0,1]$, $m\ge 0$, $0<\eta\le 1$ and $\nu>0$.
We need to estimate $|T_{odd,\nu}^m f(x) - T_{odd,\nu}^m f(x')|$ for $f\in \mathcal{C}^\eta_{[0,1]}$.
Let $n=n(x,x')$.
We consider the following three cases: \\
(a) $m>n$; From \eqref{eq:it.TO} and Lemma \ref{le:MMYlemO}-\eqref{it:lemO2}, we have
\begin{align}
|T_{odd,\nu}^m f(x) - T_{odd,\nu}^m f(x')| 
& \nonumber= |\beta_{o,m-1}^\nu(x) f(x_m) - \beta_{o,m-1}^\nu(x') f(x_m')| \\
& \nonumber \le \|f\|_{\mathcal{C}^0}\left[\beta_{o,n}^\nu(x) \beta_{o,m-n-2}^\nu(x_{n+1}) + \beta_{o,n}^\nu(x') \beta_{o,m-n-2}^\nu(x_{n+1}')\right] \\
& \label{eq:T^m(a)} \le 2\kappa_2^\nu \|f\|_{\mathcal{C}^0}\|{\beta_{e,m-n-2}}^\nu\|_{\mathcal{C}^0} |x-x'|^{\nu/2}.
\end{align}
(b) $m\le n-\delta$; We have
\begin{align*}
|T_{odd,\nu}^m f(x) - T_{odd,\nu}^m f(x')| 
& \le \beta_{o,m-1}^\nu(x) |f(x_m)-f(x_m')| + |f(x_m')| |\beta_{o,m-1}^\nu(x) - \beta_{o,m-1}^\nu(x')|.
\end{align*}
Since $x$ and $x'$ are contained in the same branch of $A_{odd}^m$,  say $J$, from Lemma~\ref{le:MMYlemO}-\eqref{it:lemO3} and \eqref{eq:beq}, we have
$$|x_m-x_m'|\le\sup_{x_\ast\in J}\left|\frac{dA_{odd}^m(x_\ast)}{dx}\right||x-x'|\le \frac{q_m^2}{(1-g)^2}|x-x'|\le \frac{(\beta_{o,m-1}(x))^{-2}}{g^2(1-g)^2} |x-x'|.$$
From Lemma~\ref{le:MMYlemO}-\eqref{it:lemO1}, \eqref{eq:l<n-del2} and \eqref{eq:beq}, we have
\begin{align*}
|\beta_{o,m-1}^\nu(x) - \beta_{o,m-1}^\nu(x')| 
& \le \nu \max\{\beta_{o,m-1}(x),\beta_{o,m-1}(x')\}^{\nu-1}|\beta_{o,m-1}(x)-\beta_{o,m-1}(x')| \\
& \le \nu \kappa_1^{\nu-1} \beta_{o,m-1}^{\nu-1}(x)q_{m-1}|x-x'| \le \nu \kappa_1^{\nu-1}g^{-1} \beta_{o,m-1}^{\nu-2}(x)|x-x'|.
\end{align*}
Combining the above three equations, we have
\begin{align*}
&|T_{odd,\nu}^m f(x) - T_{odd,\nu}^m f(x')| \\
& \le \left\{\beta_{o,m-1}^\nu(x) \left(\frac{(\beta_{o,m-1}(x))^{-2}}{g^2(1-g)^2}|x-x'|\right)^{\eta} |f|_\eta + \|f\|_{\mathcal{C}^0}\nu \kappa_1^{\nu-1} g^{-1} (\beta_{o,m-1}(x))^{\nu-2}|x-x'| \right\} \\
& = |x-x'|^{\eta}\beta_{o,m-1}^{\nu-2\eta}(x)\left\{\left(\frac{1}{g(1-g)}\right)^{2\eta}|f|_\eta + \frac{\|f\|_{\mathcal{C}^0}\nu \kappa_1^{\nu-1}}{g} \left(\frac{|x-x'|}{\beta_{o,m-1}^{2}(x)}\right)^{1-\eta} \right\}.
\end{align*}
From \eqref{eq:beq} and Lemma~\ref{le:MMYlemO}-\eqref{it:lemO3}, we have $|x-x'|(\beta_{o,m-1}(x))^{-2} \le  {(1-g)^{-2}}q_{m}^2|J|\le  {(1-g)^{-2}}g^{-1}$.
By letting $K = \left(g^{-1}(1-g)^{-1}\right)^{2\eta} + \nu \kappa_1^{\nu-1} g^{-1} (g^{-1}(1-g)^{-2})^{1-\eta}$, we have
\begin{equation}\label{eq:T^m(b)}
 |x-x'|^{-\eta}|T_{odd,\nu}^m f(x) - T_{odd,\nu}^m f(x')| \le  {K \|f\|_{\eta}} \beta_{o,m-1}^{\nu-2\eta}(x).
 \end{equation}\\
(c) $n-\delta<m\le n$; It is enough to consider the case of \ref{it:B}, \ref{it:C} and \ref{it:D}.
Then, $x$ and $x'$ belong to the adjacent branches of $A_o^{m}$.
In each case, we defined $x''$ as the common end-point of the adjacent branches.
From \eqref{eq:T^m(b)} and Lemma~\ref{le:MMYlemO}-\eqref{it:lemO1}, we have
\begin{align}
\left|T_{odd,\nu}^m f(x) - T_{odd,\nu}^m f(x')\right| 
&\nonumber \le |T_{odd,\nu}^m f(x) - T_{odd,\nu}^m f(x'')| + |T_{odd,\nu}^m f(x') - T_{odd,\nu}^m f(x'')| \\
&\nonumber \le 2 K \|f\|_\eta \max \left\{\beta_{o,m-1}^{\nu-2\eta}(x)|x-x''|^\eta, \beta_{o,m-1}^{\nu-2\eta}(x')|x'-x''|^\eta\right\} \\
&\label{eq:T^m(c)} \le 2 K \|f\|_\eta \kappa_1^{\nu-2\eta}\beta_{o,m-1}^{\nu-2\eta}(x)|x-x'|^\eta.
\end{align}

Thus we have the first assertion.
\medskip

\noindent (2) and (3):  {Let $\nu\ge 2\eta$, $f\in \mathcal{C}^\eta_{ {[0,1]}}$ and $x,~x'\in [0,1]$.
By \eqref{eq:beO}  {and \eqref{eq:it.TO}}, we have 
$$\|T_{odd,\nu}^m f\|_{\mathcal{C}^0} \le g^{ {(m-1)}\nu}\|f\|_{\mathcal{C}^0}.$$
If $m\le n$, by \eqref{eq:T^m(b)} and \eqref{eq:T^m(c)}, we have
$$|T_{odd,\nu}^mf|_\eta \le 2K \|f\|_\eta \kappa_1^{\nu-2\eta} g^{(m-1)(\nu-2\eta)}.$$
By the above arguments in (a) and \eqref{eq:beO}}, for all $m > n$, we have
$$|T_{odd,\nu}^m f(x) - T_{odd,\nu}^m f(x')| 
 \le 2\kappa_2^\nu \|f\|_{\mathcal{C}^0} |x-x'|^{\eta}|x-x'|^{\nu/2-\eta} g^{(m-n-  2)\nu}.
$$
From \eqref{eq:l<n-del3} and Lemma~\ref{le:MMYlemO}-\eqref{it:lemO1}, we have 
$|x-x'| = |x_{n-2}-x_{n-2}'|\beta_{e,n-2}(x)\beta_{e,n-2}(x') \le g^{2(n-  2)}$.
Since $(n-2)(\nu-2\eta)+(m-n-2)\nu > (m-2)(\nu-2\eta)-2\nu,$ we have
$$|T_{odd,\nu}^m f|_{\eta} \le 2\kappa_2^\nu\|f\|_\eta g^{ {(m-2)}(\nu-2\eta) {-2\nu}}.$$
Combining the above equations, for all $m\ge 0$, we have
$$\|T_{odd,\nu}^m f\|_\eta \le  \max\{2\kappa_2^\nu g^{-2\nu}, 2K\kappa_1^{\nu-2\eta}\} g^{ {(m-2)}(\nu-2\eta)}\|f\|_\eta,$$
 {which implies that} $$\rho(T_{odd,\nu})\le g^{\nu-2\eta}.$$
If $\nu>2\eta$, then $g^{\nu-2\eta}$ is less than $1$.
Thus the operator ${\bf B_{odd,\nu}}$ is well-defined on $\mathcal{C}^\eta_{[0,1]}$ and it can be represented  {by the} Neumann series as in the statement.
If $\nu=2\eta$, then we have the statement by taking $c_{3,\nu}=\max\{2\kappa_2^\nu g^{-2\nu}, 2K\}$.
\end{proof}

\begin{proof}[Proof of Theorem~\ref{th:op2}]
Assume that $\nu\le 2\eta$. Let $x,x'\in[0,1]$ and $n=n(x,x')$.
From \eqref{eq:T^m(a)}, \eqref{eq:T^m(b)} and \eqref{eq:T^m(c)}, we have
\begin{align*}
&|{\bf B_{odd,\nu}}f(x) -  {\bf B_{odd,\nu}}f(x')| 
 \le \sum_{m=0}^n |T_{odd,\nu}^m f(x)-T_{odd,\nu}^m f(x')| + \sum_{m=n+1}^\infty |T_{odd,\nu}^m f(x)-T_{odd,\nu}^m f(x')| \\
&\qquad\qquad\qquad\qquad \le c \|f\|_\eta \bigg(|x-x'|^{\eta-\nu/2}\sum_{m=0}^n  \beta_{o,m-1}(x)^{\nu-2\eta} + \sum_{m=n+1}^\infty g^{(m-n-  2)\nu} \bigg)|x-x'|^{\nu/2},
\end{align*}
where $c = \max\{2\kappa_2^\nu, 2K \kappa_1^{\nu-2\eta}g^{-\nu}\}$.
By \eqref{eq:beq} and \eqref{eq:beO}, $|x-x'|^{\eta-\nu/2}\sum_{m=0}^n (\beta_{o,m-1}(x))^{\nu-2\eta}$ is bounded above by
\begin{equation}\label{eq:sumbeo}
 \frac{ {|x-x'|^{\eta-\nu/2}}}{\beta_{o,n-1}(x)^{2\eta-\nu}} \sum_{m=0}^n(x_mx_{m+1}\cdots x_{n-1})^{2 \eta - \nu} 
 \le \frac{ {(|J|q_n^{2})^{\eta - \nu/2}}\sum_{m=0}^n g^{(n-m)(2\eta-\nu)}}{ {(1-g)^{2\eta-\nu}}}.
\end{equation}
If $x,~x'$ are in the same branch, then let $J$ be the branch of $A^n$, otherwise, i.e. $x,~x'$ are in the adjacent branches of $A^n$, then let $J$ be the union of the branches.
From Lemma~\ref{le:MMYlemO}-\eqref{it:lemO3} and \eqref{it:lemO4}, $|x-x'|q_n^2\le |J|q_n^2<2\kappa_3/g$.
If $2\eta>\nu$, then \eqref{eq:sumbeo} is uniformly bounded and 
$$|{\bf B_{odd,\nu}}f(x) - {\bf B_{odd,\nu}}f(x')|
\le c' \|f\|_\eta|x-x'|^{\nu/2},$$
where $c' = c\left(\frac{(2\kappa_3/g)^{\eta-\nu/2}}{(1-g)^{2\eta-\nu}(1-g^{2\eta-\nu})}+\frac{g^{\nu}}{1-g^\nu}\right)$,
which proves the first assertion.

If $2\eta = \nu$, then  {\eqref{eq:sumbeo} is bounded above by $n(x,x')$.}
Thus we have
$$|{\bf B_{odd,\nu}}f(x) - {\bf B_{odd,\nu}}f(x')|
\le c\|f\|_\eta (n(x,x')+1+ g^\nu(1-g^\nu)^{-1}) |x-x'|^{\nu/2}.$$
Since $x,~x'$ are in the same branch of $A_{odd}^{n-2}$, say $J'$, we have $|x-x'|\le |J'|\le  {g^{-1}}q_{n-2}^{-2}$ by Lemma~\ref{le:MMYlemO}-\eqref{it:lemO3}.
By \eqref{eq:beq}, we have $q_{n-2}^{-1} \le  {(1-g)^{-1}}\beta_{o,n-3}(x)\le  {(1-g)^{-1}g^{n- 3}}$.
Thus, we have $|x-x'|\le  {g^{2n-7}(1-g)^{-2}}$,  {which means that 
$n \le \frac{\log {|x-x'|^{-1}} + \log(g^{-1}(1-g)^{-2})}{2\log(1/g)}$.
There are constants $c_{1}''>0$ and $\kappa_2''>0$ such that}
$$|{\bf B_{odd,\nu}}f(x) - {\bf B_{odd,\nu}}f(x')| \le \|f\|_\eta  {(\kappa_1'' \log|x-x'|^{-1}+\kappa_2'')} |x-x'|^{\nu/2}.$$
 {If $\eta'<\nu/2$, then $|{\bf B_{odd,\nu}}f(x) - {\bf B_{odd,\nu}}f(x')| \le \|f\|_\eta (\kappa_1''\log|x-x'|^{-1}+\kappa_2'')|x-x'|^{\nu/2-\eta'} |x-x'|^{\eta'}$, and $(\kappa_1''\log|x-x'|^{-1}+\kappa_2'')|x-x'|^{\nu/2-\eta'}$ is uniformly bounded above.}
\end{proof}

\section{Proof of Proposition \ref{pr:est.BE} and \ref{pr:est.BOO}}\label{ap:C}

\begin{proof}[Proof of Proposition~\ref{pr:est.BE}]
For $x\in \mathbb{R}\setminus \mathbb{Q}$, let 
\begin{align*}
I_{\ast}(x) & :=\left\{n\in \mathbb{N} :x_{e,n}\in\left(1/3,1\right]\right\}=\left\{n:a_{e,n+1} = 2\right\}, \\
I_{\ast\ast}(x) & :=\left\{n\in \mathbb{N} :x_{e,n}\in\left(0,1/3\right]\right\}=\left\{n:a_{e,n+1}\ge 4\right\}.
\end{align*}
Let $x_{e,m+1},\cdots, x_{e,m+\ell}\in (1/2,1]$ and $x_{e,m},~x_{e,m+\ell+1}\not\in(1/2,1]$.
By the same argument in \cite[Proof of Thm 13]{LMNN10}, 
for $0 < t \le \ell-1$, we have
\begin{equation}\label{eq:estE1}
(x_{e,-1}x_{e,0}\cdots x^e_m x_{e,m+1} \cdots x_{e,m+t})^\nu \log \frac{1}{x_{e,m+t+1}}
<(x_{e,1-}x_{e,0}\cdots x^e_m)^\nu \frac{3}{\ell+2},
\end{equation}
and for $t=0$, we have
$(x_{e,-1}x_{e,0}\cdots x^e_m)^\nu \log{(1/{x_{e,m+1}})} < (x_{e,-1}x_{e,0}\cdots x^e_m)^\nu \ell^{-1}.$
Then, we have
$$\sum_{t=0}^{\ell-1}(x_{e,-1}x_{e,0}\cdots x_{e,m+t})^\nu \log \frac{1}{x_{e,m+t+1}}<3(x_{e,-1}x_{e,0}\cdots x^e_m)^\nu.$$
If $x_{e,m+1},\cdots,x_{e,m+\ell} \in (1/3,1/2]$ and $x_{e,m},x_{e,m+\ell+1}\not\in (1/3,1/2]$, then
\begin{equation}\label{eq:estE2}
\sum_{t=0}^{\ell-1} (x_{e,-1}x_{e,0}\cdots x_{e,m+t})^\nu \log \frac 1{x_{e,m+t+1}} 
\le \sum_{t=0}^{\ell-1} (x_{e,0}\cdots x_{e,m-1}x^e_m)^\nu \frac {\log 3}{3^t}
<(\log 3)(x_{e,0}\cdots x_{e,n-1}x_{e,n})^\nu.
\end{equation}
With a similar argument of \cite{LMNN10}, we have
$$\sum_{n\in I_{\ast}(x)} (x_{e,-1}x_{e,0}x_{e,1}\cdots x_{e,n-1})^\nu \log{\frac{1}{x_{e,n}}} < 3.$$
Then, it is sufficient to show that there exists $C_\nu>0$ such that
\begin{equation}\label{Eq1}
\left|\sum_{n\in I_{\ast\ast}(x)}\beta_{e,n-1}^\nu \log{\frac{1}{x_{e,n}}}-\sum_{n\in I_{\ast\ast}(x)}\frac{\log{q_{e,n+1}}}{(q_{e,n})^\nu}\right|<C_\nu.
\end{equation}

The left hand side of \eqref{Eq1} is bounded above by
\begin{equation}\begin{split}\label{eq:apC}
\sum_{n\in I_{\ast\ast(x)}}\left| \beta_{o,n-1}^\nu \log \beta_{o,n} q_{o,n+1} \right|
&+  \sum_{n\in I_{\ast\ast(x)}}\left| \beta_{e,n-1}^\nu \log \beta_{e,n-1} \right| \\
&+\sum_{n\in I_{\ast\ast(x)}} \left|\left\{\frac{1}{(q_n^e)^\nu} - (\beta_{e,n-1})^\nu\right\} \log q_{e,n+1}\right|.
\end{split}\end{equation}

From now on, let $n\in I_{\ast\ast}(x)$.
Since $x_{e,n}\le 1/3$ and $a_{n+1}^e\ge 4$, we have $q_{e,n-1}/q_{e,n} \le  {1/3}$ and $1+\veps_{e,n} q_{e,n-1}{(q_{e,n})^{-1}}x_{e,n} \ge 2/3$.
By the fact that \eqref{eq:qx-p} holds for the ECF, we have
\begin{equation}\label{eq:Eq2E}
\beta_{e,n-1}^\nu = \frac{1}{(q_{e,n} + \veps_{e,n}q_{e,n-1}x_{e,n})^\nu}\le \frac{3^\nu}{2^\nu (q_{e,n})^\nu},
\end{equation}
\begin{equation}\label{eq:Eq3E}
\beta_{e,n} =  \frac{1}{q_{e,n+1} + \veps_{e,n+1}q_{e,n}x_{e,n+1}}\le \frac{1}{q_{e,n+1}-q_{e,n}}\le \frac{3}{2 q_{e,n+1}}.
\end{equation}
Thus  {the first summand of \eqref{eq:apC} is bounded above by}
$$\sum_{n\in I_{\ast\ast}(x)}\beta_{e,n-1}^\nu \log\beta_{e,n}q_{e,n+1}\le \sum_{n\in I_{\ast\ast}(x)}\frac{3^\nu \log(3/2)}{2^\nu q_{e,n}^\nu}<\frac{3^\nu \log(3/2)}{4^\nu}.$$
Since $\{q_{e,n}\}$ is increasing, $1+\veps_{e,n} q_{e,n-1}(q_{e,n})^{-1}x_{e,n} \le 2$.
Thus we have 
$$(\beta_{e,n-1})^{-1} = q_{e,n}(  1+\veps_{e,n}q_{e,n-1}(q_{e,n})^{-1}x_{e,n}) \le 2q_{e,n}.$$
Note that if $n$ is the $i$th number in $I_{\ast\ast}(x)$, then we have $q_{n} > 3^{i-1}$, since $q_{n}\ge 3q_{n-1}$ if $n\in I_{\ast\ast}(x)$, and $\{q_n\}_{n\in\mathbb{N}}$ is increasing.
By \eqref{eq:Eq2E} and the fact that $\log N \le N^{\nu/2}$ for large enough $N$,  {the second summand of \eqref{eq:apC} is bounded above by}
\begin{equation}\label{eq:Eq4E}
\sum_{n\in I_{\ast\ast}(x)} \frac{3^\nu \log (2 q_{e,n})}{(2q_{e,n})^\nu} 
\le S_\nu +  \sum_{n\in I_{\ast\ast}(x)}\frac{3^\nu}{(2q_{e,n})^{\nu/2}} \le S_\nu + \frac{3^{3/2}}{6^{\nu/2}-2^{\nu/2}},
\end{equation}
 {where $S_\nu = \sum_{n\le N}\frac{3^\nu \log (2 q_{e,n})}{(2q_{e,n})^\nu}$.}
 {Since \eqref{eq:qx-p} holds for the ECF}, we have $\veps_{e,n}\beta_{e,n} q_{e,n-1}+\beta_{e,n-1}q_{e,n} = 1$.
By \eqref{eq:Eq3E}, the last summand of \eqref{eq:apC} is bounded above by
\begin{align}
&\label{eq:logq/qE} c_\nu \sum_{n\in I_{\ast\ast}(x)} \max\left\{\frac{1}{(q_{e,n})^{\nu-1}}, \beta_{e,n-1}^{\nu-1}\right\} \left|\veps_{e,n} \frac{q_{e,n-1}}{q_{e,n}} \beta_{e,n} \right|\log q_{e,n+1},
\end{align}
for some $c_\nu>0$.
If $\nu\ge 1$, then \eqref{eq:logq/qE} is bounded above by 
$$\sum_{n\in I_{\ast\ast}(x)}  \beta_{e,n}\log q_{e,n+1} 
\le \frac{3}{2} \sum_{n\in I_{\ast\ast}(x)}  \dfrac{\log q_{e,n+1}}{q_{e,n+1}} 
\lesssim\sum_{i=1}^\infty \frac{1}{\sqrt{3}^i} .$$
If $0<\nu<1$, then 
\eqref{eq:logq/qE} bounded above by
\begin{align*}
&\sum_{n\in I_{\ast\ast}(x)} \max\left\{\frac{1}{\beta_{e,n-1}^{1-\nu}}, q_{e,n}^{1-\nu}\right\} \dfrac{q_{e,n-1}}{q_{e,n}}\beta_{e,n}\log q_{e,n+1} \\
& \le \sum_{n\in I_{\ast\ast}(x)} \max \left\{\frac{\beta_{e,n}}{\beta_{e,n-1}^{1-\nu}},\frac{3q_{e,n-1}}{2q_{e,n}^\nu q_{e,n+1}}\right\}\log q_{e,n+1} 
 \le \sum_{n\in I_{\ast\ast}(x)} \max \left\{\beta_{e,n}^{\nu},\frac{3}{2 q_{e,n+1}^\nu}\right\}\log q_{e,n+1} \\
& \le \sum_{n\in I_{\ast\ast}(x)} \frac{3}{2 q_{e,n+1}^\nu}\log q_{e,n+1} \lesssim_\nu \sum_{i=1}^\infty \frac{1}{(3^{\nu/2})^i}.
\end{align*}
\end{proof}

\begin{proof}[Proof of Proposition~\ref{pr:est.BOO}]
For $x\in \mathbb{R}\setminus \mathbb{Q}$, let 
$$J_\ast(x) = \left\{ n: x_{o,n}\in\left[0,1/2\right)\right\} 
\quad \text{and} \quad J_{\ast\ast}(x)=\left\{n:x_{o,n}\in\left[1/2,1 \right]\right\}.$$
Since $\iota A_{oo} = A_{even} \iota$, if $y = \iota(x)$, then $\iota(x_{oo,n}) = y_{e,n}$.
By the arguments in the proof of Proposition~\ref{pr:est.BE}, we have
$$\sum_{n\in J_\ast (x)} \left[(x+1) \iota(x_{oo,-1})\iota(x_{oo,0})\cdots \iota(x_{oo,n-1})\right]^\nu \log\frac{1}{\iota(x_{oo,n})} < 6.$$
Thus, similarly, It is sufficient to show that there exists $C_\nu>0$ such that
\begin{equation}\label{Eq2}
\left| \frac 12 \sum_{n\in I_{\ast\ast}(x)} \beta_{oo,n-1}^\nu \log{\frac{1}{\iota(x_{oo,n})}}-\sum_{n\in I_{\ast\ast}(x)}\frac{\log{q_{oo,n+1}}}{q_{oo,n}^\nu}\right|<C_\nu.
\end{equation}
By \eqref{eq:beOO}, we have 
$\beta_{oo,n-1}q_{oo,n}+\veps_{oo,n}\beta_{oo,n} q_{oo,n-1} = 2.$
Also, we have $\iota(x_{oo,n})\le 1/3$, $q_{oo,n-1}/q_{oo,n}\le 1$ and $q_{oo,n+1}\ge 3q_{oo,n}$,
with the same argument in the proof of Proposition~\ref{pr:est.BE}, we have the conclusion.
\end{proof}

\section{The Brjuno functions as cocycles}\label{ap:cocycle}
In this section, we show how to interpret $B_{0,\nu}$, $B_{odd,\nu}$ and $B_{even,\nu}$ as cocycles under the actions on the real line respectively of the modular group $\mathrm{PSL}_2(\mathbb{Z})$, $\Gamma$ and $\Theta$. This extends to these functions the results obtained for the Brjuno function in Appendix 5 of \cite{MMY01}.
Let $G$ be a group and $G$ act on the left on a set $X$.
Let $A$ be an abelian ring, $A^*$ be the multiplicative group of invertible elements of $A$, and $M$ be a $A$-module.
A natural $G$-action on $M^X$ equipped with \emph{an automorphic factor} $\chi:G\times X \to A$ is given by
$$g.\varphi (x) = \chi(g^{-1},x)\varphi(g^{-1}.x) \text{ for all } g\in G, \varphi \in M^X, x\in X.$$
The automorphic factor should satisfies 
\begin{equation}\label{eq:auto}\chi(g_1g_0,x) = \chi(g_1,g_0x)\chi(g_0,x).\end{equation}
Then $M^X$ has a $\mathbb{Z}^{[G]}$-module structure.

Let $C^n:= C^n(G,M^X) = \mathrm{Map}(G,M^X)$. We consider the cochain complex 
$$\cdots \to 0\to C^0 ~\text{\raisebox{1ex}{$\underrightarrow{~~d^0~~}$}}~  C^1 ~\text{\raisebox{1ex}{$\underrightarrow{~~d^1~~}$}}~ C^2 ~\text{\raisebox{1ex}{$\underrightarrow{~~d^2~~}$}}~ \cdots  ~\text{\raisebox{1ex}{$\underrightarrow{~~d^{n-1}~~}$}}~C^n ~\text{\raisebox{1ex}{$\underrightarrow{~~d^n~~}$}}~ \cdots $$ defined by
\begin{align*}
(d^n \varphi) (g_0,\dots g_n) & = g_0.\varphi(g_1,\dots,g_n)+\sum_{i=0}^{n-1}(-1)^{i+1}\varphi(g_0,\dots,g_ig_{i+1},\dots , g_n)\\
& + (-1)^{n+1}\varphi(g_0,\dots,g_{n-1}).
\end{align*}
An $n$-coboundary is an element of the image of $d^{n-1}$.
An $n$-cocycle is an element of the kernel of $d^n$.
We identify $C^0$ with $M^X$.
If $f$ is a $1$-coboundary of $\varphi \in M^X$, then
\begin{equation}\label{eq:cob}
f(g)(x) = (d^0 \varphi)(g)(x) = g.\varphi(x) - \varphi(x) = \chi(g^{-1},x)\varphi(g^{-1}.x)-\varphi(x).
\end{equation}
A $1$-cocycle $c$ satisfies
$$(d^1 c)(g_0,g_1) = g_0.c(g_1)-c(g_0g_1)+c(g_0) = 0.$$
Thus, $c(g_0g_1) = c(g_0) + g_0.c(g_1)$.
Let $\check{c}(g) = c(g^{-1})$.
Then, we have
\begin{equation}\label{eq:coc}
\check{c}(g_0g_1;x)
 = \check{c}(g_1;x) + \chi(g_1,x)\check{c}(g_0;g_1.x).
 \end{equation}

\subsection{Semi-Brjuno function and the modular group action}
The modular group 
$\mathrm{PSL}_2(\mathbb{Z})$ is generated by $T(x)=x+1$ and $\sigma(x)=-1/x$.

\begin{prop}\label{pr:uni.b}
For $g\in \mathrm{PSL}_2(\mathbb{Z})\setminus\{I\}$ and $x_0\in\mathbb{R}\setminus\mathbb{Q}$,
there exist unique integer $r\ge 1$ and $g_1,\cdots,g_r \in\{T,T^{-1},\sigma\}$ such that $g= g_r\cdots g_1$ with $g_ig_{i-1}\not= I$, and 
\begin{align}
& \label{eq:rep.b}  x_{i-1}\in (-\infty,-1)\cup (0,1) \text{ if } g_i = \sigma, 
\end{align}
where $x_{i}=g_i g_{i-1} \cdots g_1x_0$.
\end{prop}

\begin{proof}
We have $\sigma = T\sigma T\sigma T$.
For $g=\sigma$, $r$ and $g_i$'s are determined as follows.
\begin{enumerate}
\item If $x_0\in (-1,-\frac{1}{2})$, then $r=5$ and $g_5\cdots g_1 = T\sigma T\sigma T$
since $T(x_0)\in (0,\frac{1}{2})$ and $T\sigma T(x_0)\in (-\infty,-1)$.
\item If $x_0\in (-\frac{1}{2},-\frac{1}{3})$, then $r = 9$ and $g_9(g_8\cdots g_4)g_3g_2g_1 = T(T\sigma T\sigma T )T\sigma T$
since $T\sigma T(x_0)\in (-1-\frac{1}{2})$.
\item Inductively, if $x_0\in (-\frac{1}{n},-\frac{1}{n+1})$, then $r = 4n+1$ and
$g_{r}(\prod_{i=4}^{r-1} g_i)g_3g_2g_1 = T(h_{r'}\cdots h_1)T\sigma T,$
where $h_{r'}\cdots h_1$ is the representation of $\sigma$ for $x\in(-\frac{1}{n-1},-\frac{1}{n})$.
\item For $x_0\in (n,n+1)$, by using $\sigma=\sigma^{-1}$, we have $r= 4n+1$ and 
$g_r\cdots g_1 = T^{-1}\sigma T^{-1}(h_1^{-1}\cdots h_{r'}^{-1})T^{-1}.$
\end{enumerate}

To show the uniqueness, assume that 
%\begin{equation}\label{eq:rep.b2}
$g  = g_r\cdots g_1 = I.$
%\end{equation}
Let $i_1<i_2<\cdots<i_k$ be indices such that $g_{i_\ell}=\sigma$.
Since $1 = g'(x_0) = \prod_{i=1}^r g'_i(x_{i-1}) = \prod_{\ell=1}^k \frac{1}{x_{i_\ell-1}^2},$ we have
\begin{equation}\label{eq:rep.b3}
|x_{i_1-1}x_{i_2-1}\cdots x_{i_k-1}| = 1.
\end{equation}
If $|x_{i_\ell-1}|>1$, then $x_{i_\ell-1}<-1$ by Condition~\eqref{eq:rep.b}.
Then, $x_{i_\ell}\in (0,1)$ and $x_{i_{\ell+1}-1} = x_{i_\ell}\pm(i_{\ell+1}-i_\ell-1)\in \mathbb{R}\setminus(0,1)$.
By Condition~\eqref{eq:rep.b}, $x_{i_{\ell+1}-1}<-1$.
Since $g_{i-1}\cdots g_1 g_r \cdots g_i = I$ for all $i$, if $|x_{i_\ell-1}|>1$ for some $\ell$, then $|x_{i_\ell-1}|>1$ for all $\ell$.
By \eqref{eq:rep.b3}, it cannot happen, thus we have $|x_{i_\ell-1}|\le 1$ for all $\ell$. If $|x_{i_\ell-1}|<1$ for some $\ell$, then it contradicts to \eqref{eq:rep.b3}
Therefore $x_{i_\ell-1}=1$ for all $\ell$. It contradicts to $x_0$ is irrational.
Thus, there are no $g_i$'s such that $g_r\cdots g_1=I$.
\end{proof}

\begin{coro}\label{co:auto.b}
Let $A$ be an abelian ring. For two maps
$t:\mathbb{R}\setminus\mathbb{Q}\to A^\ast\text{ and }s:(0,1)\setminus\mathbb{Q}\to A^\ast,$
there exists a unique automorphic factor $\chi$ such that
\begin{equation*}\begin{matrix*}[l]
\chi(T,x) = t(x) & \text{for }x\in\mathbb{R}\setminus\mathbb{Q},  \vspace{0.3ex} \\
\chi(\sigma,x) = s(x) & \text{for }x\in(0,1)\setminus\mathbb{Q}.
\end{matrix*}\end{equation*}
\end{coro}

\begin{proof}
We define by 
\begin{equation}\begin{matrix*}[l] \label{eq:auto.b}
\chi(T^{-1},x) = (t(x-1))^{-1} & \text{for }x\in\mathbb{R}\setminus\mathbb{Q}, \text{ and } \vspace{0.3ex} \\
\chi(\sigma,x) = (s(-1/x))^{-1} & \text{for }x\in(-\infty,-1)\setminus\mathbb{Q}.
\end{matrix*}\end{equation}
Then it follows \eqref{eq:auto}.
From Proposition~\ref{pr:uni.b}, for any $g\in \mathrm{PSL}_2(\mathbb{Z})$ and $x_0\in \mathbb{R}\setminus\mathbb{Q}$, there exists $g_i$'s such that $g= g_r\cdots g_1$ following \eqref{eq:rep.b}, by \eqref{eq:auto}, if there exists such an automorphic factor $\chi$, then 
$$\chi(g,x_0) = \prod_{i=1}^r \chi(g_i,g_{i-1}\cdots g_1 x_0),$$
which means that $\chi$ is uniquely determined.

To show the existence, it is enough to show that 
$\chi(gh,x_0) = \chi(g,hx_0)\chi(h,x_0).$
Let 
\begin{equation}\label{eq:set.gh}
g=g_M\cdots g_1 \text{ and }h = h_N\cdots h_1
\end{equation}
be the representations of  $g$ for $hx_0$ and $h$ for $x_0$ following \eqref{eq:rep.b}, respectively.
From the uniqueness of the representations, the representation of $gh$ for $x_0$ is 
$gh = g_M\cdots g_{i+1} h_{N-i} \cdots h_1,$
where $g_\ell = h_{N-\ell}^{-1}$ for $1\le\ell\le i$.
From \eqref{eq:auto.b},  we have $\chi(g_{\ell},h_{N-\ell}y)\chi(h_{N-\ell},y)=1$ for $1\le \ell \le i$, where $y=h_{N-\ell-1}\cdots h_1x_0$.
\end{proof}

\begin{coro}\label{co:coc.b}
Let $A$ be an abelian ring, $\chi$ be an automorphic factor, $M$ be an $A$-module such that having the $\mathbb{Z}^{[\mathrm{PSL}_2(\mathbb{Z})]}$-module structure associated with $\chi$. For two maps
$\check{c}_T: \mathbb{R}\setminus\mathbb{Q} \to M$ and $\check{c}_\sigma:(0,1) \setminus\mathbb{Q}\to M,$
there exists a unique cocycle $\check{c}:\mathrm{PSL}_2(\mathbb{Z})\times (\mathbb{R}\setminus\mathbb{Q}) \to M$ such that
\begin{equation*}\begin{matrix*}[l]
 \check{c}(T;x)=\check{c}_T(x)& \text{for }x\in \mathbb{R}\setminus\mathbb{Q}, \text{ and } \vspace{0.4ex} \\
 \check{c}(\sigma;x)=\check{c}_\sigma(x) & \text{for }x\in (0,1)\setminus \mathbb{Q}.
\end{matrix*}
\end{equation*}
\end{coro}

\begin{proof}
We define by 
\begin{equation}\label{eq:coc.b2}\begin{matrix*}[l] 
\check{c}(T^{-1};x) = -\chi(T^{-1},x)\check{c}(T;x-1) & \text{for }x\in\mathbb{R}\setminus\mathbb{Q}, \text{ and } \vspace{0.4ex}\\
\check{c}(\sigma;x) = -\chi(\sigma,x)\check{c}(\sigma;-1/x) & \text{for }x\in (-\infty,-1)\setminus\mathbb{Q}.
\end{matrix*}\end{equation}
It follows \eqref{eq:coc}.
For any $g\in \mathrm{PSL}_2(\mathbb{Z})$ and $x_0\in\mathbb{R}$, there exist $g_i$'s in $\{T,T^{-1},\sigma\}$  such that  $g = g_r\cdots g_1$ following \eqref{eq:rep.b}, 
we have
$$\check{c}(g;x_0) = \sum_{i=1}^r \chi(g_{i-1}\cdots g_1, x_0) \check{c}(g_i;g_{i-1}\cdots g_1 x_0),$$
which means that $\check{c}$ is uniquely determined.

To show the existence, it is enough to show that $\check{c}$ satisfies
$\check{c}(gh;x_0) = \check{c}(h;x_0)+\chi(h,x_0)\check{c}(g;hx_0).$
With the same set up as \eqref{eq:set.gh}, from \eqref{eq:coc.b2} and \eqref{eq:coc}, we have 
$$\check{c}(g_{\ell};h_{N-\ell}y)+\check{c}(h_{N-\ell};y)=0\text{, where }y=h_{N-\ell-1}\cdots h_1x_0$$
then we are done.
\end{proof}

Let $A=\mathbb{R}$, $t(x)=1$, $s(x) = x^\nu$ on $(0,1)$ for $\nu\in\mathbb{R}$, $\check{c}_T(x) = 0$ and $\check{c}_\sigma(x) = -\log{x}$ on $(0,1)$.
There exist an automorphic factor $\chi$ and a cocycle $\check{c}$ such that
\begin{equation*}\begin{cases}
\chi(T^n,x) = 1 &\text{for }n\in \mathbb{Z},~x\in\mathbb{R}\setminus\mathbb{Q},\\
\chi(\sigma,x) = x^\nu &\text{for } x\in (0,1)\setminus\mathbb{Q},
\end{cases} \text{ and }
\begin{cases}
\check{c}(T;x) = 0 & \text{for }x\in\mathbb{R}\setminus\mathbb{Q},\\
\check{c}(\sigma;x) = -\log{x} &\text{for }x\in (0,1)\setminus\mathbb{Q}.
\end{cases}
\end{equation*}
From \eqref{eq:ft.eq.B}, $B_{0,\nu}(x) = B_{0,\nu}(x+1)$ for $\mathbb{R}\setminus\mathbb{Q}$ and \eqref{eq:cob}, we have
\begin{equation*}\begin{cases}
(d^0B_{0,\nu})(\sigma;x) = \chi(\sigma,x)B_{0,\nu}(\sigma x) - B_{0,\nu}(x) =  \log x & \text{for }x\in (0,1)\setminus\mathbb{Q},  \\
(d^0B_{0,\nu})(T;x) = \chi(T^{-1},x)B_{0,\nu}(T^{-1}(x)) - B_{0,\nu}(x)=0 & \text{for } x\in \mathbb{R}\setminus\mathbb{Q}.\\
\end{cases}\end{equation*}
By the uniqueness of the cocycle $\check{c}$, we have $\check{c} = d^0(-B_{0,\nu})$, i.e. the $1$-coboundary of $-B_{0,\nu}$ is a $1$-cocycle under the $\mathrm{PSL}_2(\mathbb{Z})$-action.

\subsection{OCF-Brjuno function and the $\Gamma$-action}
From now on, let us denote by $U(x) = -x$.
Let $\widetilde{\Gamma} = \Gamma\sqcup U \Gamma$, i.e.
$$\widetilde{\Gamma} =\left\{\begin{pmatrix}a&b\\c&d\end{pmatrix}\in\mathrm{PGL}_2(\mathbb{Z}):\begin{pmatrix}a&b\\c&d\end{pmatrix}\equiv 
\begin{pmatrix}1&0\\0&1\end{pmatrix}, \begin{pmatrix}1&1\\1&0\end{pmatrix}\text{ or } 
\begin{pmatrix}0&1\\1&1\end{pmatrix} \text{ (mod 2)}\right\}.$$
The group $\widetilde\Gamma$ is generated by 
%$$\tau = \begin{pmatrix}1&2\\0&1 \end{pmatrix}, ~
%\Delta = \begin{pmatrix}1&-1\\1&0\end{pmatrix} 
%\text{ and }U = \begin{pmatrix}-1&0\\0&1\end{pmatrix}.$$
%We denote by 
$\tau(x) = x+2$, $\Delta(x) = 1-1/x$ and $U(x)=-x$.

\begin{prop}\label{pr:dec.o}
For all $g\in \widetilde\Gamma\setminus\{I\}$ and $x_0\in\mathbb{R}\setminus \mathbb{Q}$, there exist unique integer $r\ge 1$ and $g_1,\dots,g_r\in\{\tau, \tau^{-1},\Delta, \Delta^{-1},U\}$ such that $g=g_r\dots g_1$ with $g_ig_{i-1}\not =I$
\begin{equation}\label{eq:uni.o}\begin{matrix*}[l]
x_{i-1} \in (-1,1)=:J_U              & \text{if }g_i=U,                         \vspace{0.4ex} \\
x_{i-1} \in (1,\infty)=:J_\Delta   & \text{if }g_i=\Delta \text{ and } \vspace{0.4ex} \\
x_{i-1}\in(0,1)=:J_{\Delta^{-1}} & \text{if }g_i=\Delta^{-1}, 
\end{matrix*}\end{equation}
where $x_i:=g_ig_{i-1}\cdots g_1 x_0$.
\end{prop}

\begin{proof}
Since $\Delta^3=I$, $U = \tau^{\pm1} U\tau^{\pm1}$,  $\tau^{-1}\Delta(x) = -1-\frac{1}{x} = U \Delta U(x)$, and $(\tau^{-1}\Delta)^3=I$, 
we have the following relations:
\begin{align}
& \label{eq:dec.o0} \Delta = \Delta^{-2} \text{ and } \Delta^{-1} =\Delta^2, \\
& \label{eq:dec.o2} \Delta = \tau U \Delta U \text{ and } \Delta^{-1} = U \Delta^{-1} U \tau^{-1}, \\
& \label{eq:dec.o4} \Delta = \tau \Delta^{-1} \tau \Delta^{-1} \tau \text{ and }\Delta^{-1} = \tau^{-1}\Delta \tau^{-1} \Delta \tau^{-1}.
\end{align}

\begin{enumerate}
\item\label{eq:dec.o.U} 
If $g = U$ and $x_0\in (2n-1,2n+1)$ for $n\in \mathbb Z$, then $r=2n+1$ and 
$g_{2n+1}\cdots g_1 = \tau^{-n}U\tau^{-n}.$
\item\label{eq:dec.o6}
For $g = \Delta$ and $x_0\in (-\infty,-1)$, we will use \eqref{eq:dec.o2}.
Since $U x_0 \in (1,\infty) =J_\Delta$ and $\Delta U x_0 \in (\frac{1}{2},1)\subset J_U$, by \eqref{eq:dec.o.U}, for $x_0\in (-2n-1,-2n+1)$,  $n\in\mathbb{N}$, we have $r = 2n+4$ and $g_r\cdots g_1 = \tau U \Delta \tau^n U \tau^n.$
\item\label{eq:dec.o6-1}
For $g=\Delta^{-1}$ and $x_0 \in (1,2)$, we will use \eqref{eq:dec.o2}.
We have $\tau^{-1}x_0 \in (-1,0)=J_U$, $U\tau^{-1}x_0 \in (0,1)=J_{\Delta^{-1}}$ and $\Delta^{-1}U\tau^{-1}x_0 \in (1,\infty)$. 
Since $\Delta^{-1}U\tau^{-1}x_0 \in (2n-1,2n+1)$ for some $n\in\mathbb{N}$ if and only if $x_0\in \left(\frac{2n+2}{2n+1},\frac{2n}{2n-1}\right)$, 
by \eqref{eq:dec.o.U}, for $x_0\in \left(\frac{2n+2}{2n+1},\frac{2n}{2n-1}\right)$, we have $r=2n+4$ and 
$g_r\cdots g_1 = \tau^{-n}U\tau^{-n}\Delta^{-1}U\tau^{-1}.$
\item\label{eq:dec.o7} 
For $g=\Delta^{-1}$ and $x_0\in (-\infty,-1)$, we will use \eqref{eq:dec.o0} and \eqref{eq:dec.o6}.
By \eqref{eq:dec.o6}, we have a decomposition of $\Delta$ on $(-2n-1,-2n+1)$ as $\tau U \Delta \tau^n U \tau^n$.
Since $\Delta x_0 \in (1,2)\subset J_\Delta$, for $x\in (-2n-1,-2n+1)$, $n\in \mathbb{N}$, we have $r = 2n+5$ and 
$g_r\cdots g_1= \Delta \tau U \Delta \tau^n U \tau^n.$
\item\label{eq:dec.o8}
For $g= \Delta^{-1}$ and $x_0\in (3,\infty)$, we will use \eqref{eq:dec.o2}. % and \eqref{eq:dec.o7}. 
By \eqref{eq:dec.o0} and \eqref{eq:dec.o2}, we have
$\Delta^{-1} U = \Delta(\Delta U) =\Delta \tau U \Delta$.
Thus, 
\begin{equation}\label{eq:dec.o8-1}
\Delta^{-1} = U(\Delta^{-1}U)\tau^{-1} = U \Delta \tau U \Delta\tau^{-1}.
\end{equation}
Since $\tau^{-1}x_0 \in (1,\infty)=J_\Delta$, $\Delta\tau^{-1}x_0 \in (0,1)\subset J_U$, $\tau U \Delta \tau^{-1} x_0 \in (1,2)\subset J_\Delta$, $\Delta \tau U \Delta \tau^{-1}x_0 \in (0,\frac{1}{2})\subset J_U$, we have $r = 6$ and 
$g_r\cdots g_1 = U \Delta \tau U \Delta \tau^{-1}.$
\item\label{eq:dec.o9}
For $g = \Delta$ and $x_0 \in (0,\frac{1}{2})$, we will use \eqref{eq:dec.o0} and \eqref{eq:dec.o6-1}.
Since $(0,\frac{1}{2})\subset J_{\Delta^{-1}}$ and $\Delta^{-1}(0,\frac{1}{2}) = (1,2)$, 
by using $\Delta = \Delta^{-1}\Delta^{-1} = (\tau^{-n}U\tau^{-n}\Delta^{-1}U\tau^{-1})\Delta^{-1}$, 
for $x\in \left(\frac{1}{2n+2},\frac{1}{2n}\right)$, we have $r=2n+5$ and
$g_r\cdots g_1 =  \tau^{-n}U\tau^{-n}\Delta^{-1}U\tau^{-1}\Delta^{-1}.$
\item\label{eq:dec.o10}
For $g=\Delta$ and $x_0\in (-\frac{1}{2},0)$, we will use \eqref{eq:dec.o8-1}.
By \eqref{eq:dec.o8-1}, we have $\Delta = \tau \Delta^{-1} U \tau^{-1} \Delta^{-1} U$.
Since $(-\frac{1}{2},0)\subset J_U$, $Ux_0 \in (0,\frac{1}{2})\subset J_{\Delta^{-1}}$, $\tau^{-1}\Delta^{-1}Ux_0 \in (-1,0)\subset J_U$ and $U\tau^{-1}\Delta^{-1}Ux_0\in (0,1)=J_{\Delta^{-1}}$, 
we have $r= 6$ and
$g_r\cdots g_1 = \tau \Delta^{-1}U\tau^{-1}\Delta^{-1} U.$
\end{enumerate}
In summary, we have a representation of $\Delta$ on $D_1:=(-\infty,-1)\cup (-\frac{1}{2},\frac{1}{2}) \cup (1,\infty)$ and
we have a representation of $\Delta^{-1}$ on $D_2:=(-\infty,-1)\cup (0,2)\cup (3,\infty)$ following \eqref{eq:uni.o}.

We will show that $\Delta$ has a representation on $(-1,-\frac{1}{2})\cup (\frac{1}{2},1)$ following \eqref{eq:uni.o} inductively.
If $\Delta$ has a representation on an interval $(a,b)\subset (-1,-\frac{1}{2})$, then $\Delta^{-1}$ has a representation on $(a,b)$ with
$\Delta^{-1} = \Delta^2$ since $\Delta(a,b)\subset (2,3)\subset D_1$.
Then, $\Delta$ has a representation on $J = \tau^{-1} \Delta \tau^{-1} (a,b)$ since $\Delta = \tau \Delta^{-1}\tau \Delta^{-1} \tau$ and $\tau J = \Delta \tau^{-1}(a,b) \subset (\frac{4}{3},\frac{7}{5})\subset D_2$.
Let us define a sequence $a_i$ and $b_i$ by
\begin{equation*}\begin{cases}
a_{i+1} = \tau^{-1}\Delta \tau^{-1} a_i = -\frac{1-a_i}{2-a_i}, \text{ where } a_0=-1,\\
b_{i+1} = - \frac{1-b_i}{2-b_i}, \text{ where }b_0 = -\frac{1}{2}.
\end{cases}\end{equation*}
Then, if $\Delta$ has a representation on $(a_i,a_{i+1})$ and $(b_{i+1},b_i)$, then $\Delta$ has a representation on $(a_{i+1},a_{i+2})$ and $(b_{i+2},b_{i+1})$.

For $g = \Delta$ and $x_0\in (-1,-\frac{2}{3})$, we have $\tau x_0 \in (1,\frac{4}{3})\subset D_1$ and 
$\tau \Delta^{-1} \tau x_0 \in (-\infty, -1)\subset D_1$. Thus, $\Delta$ on $(a_0,a_1)=(-1,-\frac{2}{3})$ has a representation by using \eqref{eq:dec.o4}.
For $g = \Delta^{-1}$ and $x_0\in (2,\frac{5}{2})$, we have $\tau^{-1}x_0\in (0,\frac{1}{2})\subset D_1$, $\Delta\tau^{-1}x_0 \in (-\infty, -1)$, $\tau^{-1}\Delta \tau^{-1}x_0 \in (-\infty, -3)\subset D_1$.
Thus, $\Delta^{-1}$ has a representation on $(2,\frac{5}{2})$ by \eqref{eq:dec.o4}.
By \eqref{eq:dec.o0}, $\Delta$ has a representation on $\Delta(2,\frac{5}{2}) = (\frac{1}{2},\frac{3}{5})$.
By \eqref{eq:dec.o2}, $\Delta$ has a representation on $(b_1,b_0)=(-\frac{3}{5},-\frac{1}{2})$.
Since $a_i\to - \frac{\sqrt{5}-1}{2}$ and $b_i\to -\frac{\sqrt{5}-1}{2}$ as $i\to \infty$, $\Delta$ has a representation on $(-1,-\frac{1}{2})$.
By using \eqref{eq:dec.o2}, we have a representation of $\Delta$ on $(\frac{1}{2},1)$.

Finally, by using \eqref{eq:dec.o0}, we can show that $\Delta^{-1}$ has a representation on $(-1,0)\cup (2,3)$.

To show the uniqueness,
we assume that there exist $g_i\in \{\tau,\tau^{-1},\Delta,\Delta^{-1},U\}$ for $1\le i\le r$ following \eqref{eq:uni.o} 
such that $g_r\cdots g_1 = I$.
Let $i_1<i_2<\cdots <i_k$ be the indices such that $g_{i_\ell}=\Delta$ or $\Delta^{-1}$.
Then we have 
\begin{equation}\label{eq:prod.o}
\prod_{i=1}^r g_i'(x_{i-1}) = 1.
\end{equation}
If $x_{i_\ell-1}>1$, then $x_{i_\ell} = 1-\frac{1}{x_{i_\ell-1}}\in (0,1)$.
Then, $g_{i_\ell+1} = U$, $\tau$ or $\tau^{-1}$.
If $g_{i_\ell+1} = \tau^{-1}$, then $g_j=\tau^{-1}$ for all $j> i_\ell$. Then $\ell=k$.
If $g_{i_\ell +1} = \tau$, then $g_{i_\ell+j} = \tau$ for all $j< i_{\ell+1}-i_\ell$. Then $x_{i_{\ell+1}-1}>1$ and $g_{i_{\ell+1}} = \Delta$.
If $g_{i_\ell+1} = U$ ,then $x_{i_\ell+1} \in (-1,0)$. Then $g_j=\tau^{-1}$ for all $j>i_\ell+1$.
Since $g_{i-1}\cdots g_1 g_r\cdots g_i=I$ for all $i$ is also following \eqref{eq:uni.o}, we conclude that if $x_{i_\ell-1}>1$ for some $\ell$, then $x_{i_\ell-1}>1$ and $g_{i_\ell}=\Delta$ for all $\ell$.
Then $\prod_{i=1}^r g_{i}'(x_{i-1}) = \prod_{\ell=1}^k x_{i_{\ell}-1}^{-2}$. It contradicts to \eqref{eq:prod.o}.
Therefore, $x_{i_\ell-1}\in (0,1)$ and $g_{i_\ell}=\Delta^{-1}$ for all $\ell$.
Thus, $\prod_{i=1}^r g_i'(x_{i-1}) = \prod_{\ell=1}^k (1-x_{i_\ell-1})^{-2} = 1.$
It contradicts to $|1-x_{i_\ell-1}|<1$.
\end{proof}

\begin{coro}\label{co:auto.o}
Let $A$ be an abelian ring. For three maps
$t:\mathbb{R}\setminus\mathbb{Q}\to A^\ast$, $s:(0,1)\setminus\mathbb{Q}\to A^\ast$ and $u:(0,1)\setminus\mathbb{Q}\to A^\ast,$
there exists a unique automorphic factor $\chi$ such that
\begin{equation*}\begin{matrix*}[l]
\chi(\tau,x) = t(x) & \text{for }x\in\mathbb{R}\setminus\mathbb{Q},\vspace{0.4ex} \\
\chi(\Delta^{-1},x) = s(x) & \text{for }x\in(0,1)\setminus\mathbb{Q}, \vspace{0.4ex} \\ 
\chi(U,x) = u(x) & \text{for }x\in (0,1)\setminus \mathbb{Q}.
\end{matrix*}\end{equation*}
\end{coro}

\begin{proof}
We define by 
\begin{equation*}\label{eq:auto.o}\begin{matrix*}[l]
\chi(\tau^{-1},x) =  (t(x-2))^{-1}              & \text{for }x\in\mathbb{R}\setminus\mathbb{Q},            \vspace{0.4ex} \\
\chi(U,x) = (u(-x))^{-1},                          & \text{for }x\in (-1,0)\setminus \mathbb{Q}, \text{ and } \vspace{0.4ex} \\
\chi(\Delta,x) = (s\left(1-1/x\right))^{-1}, & \text{for }x\in (1,\infty)\setminus\mathbb{Q}. 
\end{matrix*}\end{equation*}
From Proposition~\ref{pr:dec.o}, with the same arguments as in the proof of Corollary~\ref{co:auto.b}, we have the existence and uniqueness of such an automorphic factor.
\end{proof}

\begin{coro}\label{co:coc.o}
Let $A$ be an abelian ring, $\chi$ be an automorphic factor and $M$ be an $A$-module such that having the $\mathbb{Z}^{[\widetilde{\Gamma}]}$-module structure associated with $\chi$. For three maps
$\check{c}_\tau: \mathbb{R}\setminus\mathbb{Q} \to M$, $\check{c}_{\Delta^{-1}}:(0,1) \setminus\mathbb{Q}\to M$
 and $\check{c}_U: (0,1)\setminus \mathbb{Q} \to M,$
there exists a unique cocycle $\check{c}:\widetilde\Gamma\times (\mathbb{R}\setminus\mathbb{Q}) \to M$ such that
\begin{equation*}\begin{matrix*}[l]
\check{c}(\tau;x)=\check{c}_\tau(x)                         & \text{for }x\in \mathbb{R}\setminus\mathbb{Q}, \vspace{0.4ex} \\
\check{c}(\Delta^{-1};x)=\check{c}_{\Delta^{-1}}(x)  & \text{for }x\in (0,1)\setminus \mathbb{Q},         \vspace{0.4ex} \\
\check{c}(U;x)=\check{c}_U(x)                                & \text{for }x\in (0,1)\setminus \mathbb{Q}.
\end{matrix*}\end{equation*}
\end{coro}

\begin{proof}
We define by 
\begin{equation*}\label{eq:coc.b2}\begin{matrix*}[l]
\check{c}(\tau^{-1};x) =  -(t(x-2))^{-1}\check{c}_\tau(x-2) & \text{for }x\in\mathbb{R}\!\setminus\!\mathbb{Q},\vspace{0.4ex}\\
\check{c}(U;x) = -(u(-x))^{-1}\check{c}_U(-x)                   & \text{for }x\in(-1,0)\!\setminus\!\mathbb{Q}, \text{ and } \vspace{0.4ex} \\
\check{c}(\Delta;x) = -(s\left(1-1/x\right))^{-1}\check{c}_{\Delta^{-1}}\!\left(1-1/x\right) & \text{for }x\in (1,\infty)\!\setminus\!\mathbb{Q}. \\
\end{matrix*}\end{equation*}
From Proposition~\ref{pr:dec.o}, with the same arguments as in the proof of Corollary~\ref{co:coc.b}, we have the existence and the uniqueness of such a cocycle.
\end{proof}

Let $A = \mathbb{R}$, $t(x) = 1$, $s(x) = x^\nu$ on $(0,1)$ for $\nu \in \mathbb{R}$, $u(x) = 1$, $\check{c}_{\tau}(x)=0$, $\check{c}_{\Delta^{-1}}(x) = -\log x$ on $(0,1)$ and $\check{c}_U(x)=0$ on $(0,1)$.
There exist a unique automorphic factor $\chi$ and a unique cocycle $\check{c}$ such that
%\begin{align*}
%& \chi(\tau^n,x) = 1\text{ for all } n\in \mathbb{Z}, x\in \mathbb{R}\setminus\mathbb{Q}, \\
%& \chi(\Delta^{-1}, x) = x^\nu \text{ for all } x\in (0,1)\setminus\mathbb{Q}, \\
%& \chi(U,x) = 1 \text{ for all }x\in (0,1)\setminus\mathbb{Q}, \\
%& \check{c}(\tau;x) = 0 \text{ for all } x\in\mathbb{R}\setminus\mathbb{Q},\\
%& \check{c}(\Delta;x) = -\log x \text{ for all }x\in (0,1)\setminus\mathbb{Q},\\
%& \check{c}(U;x) = 0 \text{ for all } x\in (0,1)\setminus \mathbb{Q}.
%\end{align*}
\begin{align*}
\begin{cases}
\chi(\tau^n,x) = 1 & \text{for } n\in \mathbb{Z},~ x\in \mathbb{R}\setminus\mathbb{Q}, \\
\chi(\Delta^{-1}, x) = x^\nu & \text{for } x\in (0,1)\setminus\mathbb{Q}, \\
\chi(U,x) = 1 & \text{for }x\in (0,1)\setminus\mathbb{Q}, 
\end{cases}
\text{ and }
\begin{cases}
\check{c}(\tau;x) = 0 & \text{for } x\in\mathbb{R}\setminus\mathbb{Q},\\
\check{c}(\Delta^{-1};x) = -\log x & \text{for }x\in (0,1)\setminus\mathbb{Q},\\
\check{c}(U;x) = 0 & \text{for } x\in (0,1)\setminus \mathbb{Q}.
\end{cases}
\end{align*}
By \eqref{eq:even2Z}, \eqref{eq:fun.o} and \eqref{eq:cob}, the $1$-coboundary of $B_{odd,\nu}$ satisfies
\begin{align*}\begin{cases}
(d^0 B_{odd,\nu})(\Delta^{-1};x) = \chi(\Delta,x) B_{odd,\nu}(\Delta x) - B_{odd,\nu}(x) = \log x & \text{for }x\in (0,1)\setminus\mathbb Q,\\
(d^0 B_{odd,\nu})(U;x) = \chi(U,x) B_{odd,\nu}(U x) - B_{odd,\nu}(x) = 0 & \text{for }x\in (0,1)\setminus\mathbb Q,\\
(d^0 B_{odd,\nu})(\tau;x) = \chi(\tau^{-1},x)B_{odd,\nu}(\tau^{-1}x)-B_{odd,\nu}(x) = 0 & \text{for }x\in\mathbb R\setminus \mathbb Q.
\end{cases}
\end{align*}
Thus $\check{c} = d^0(-B_{odd,\nu})$, i.e. the $1$-coboundary of $-B_{odd,\nu}$ is a $1$-cocycle under the $\widetilde\Gamma$-action.

\subsection{ECF-Brjuno function and the $\Theta$-action}
Let $\widetilde{\Theta} = \Theta \sqcup U\Theta$ and $\Theta$, i.e.
$$\widetilde{\Theta} = \left\{\begin{pmatrix}a&b\\c&d\end{pmatrix}\in \mathrm{PGL}_2(\mathbb{Z}):\begin{pmatrix}a&b\\c&d\end{pmatrix}\equiv \begin{pmatrix}1&0\\0&1\end{pmatrix}\text{ or }\begin{pmatrix}0&1\\1&0\end{pmatrix}\text{ mod }2\right\}.$$
The group $\widetilde\Theta$ is generated by $\tau(x)=x+2$, $\sigma(x)=-1/x$ and $U(x)=-x$.

\begin{prop}
For $g\in\Theta\setminus\{I\}$ and $x_0\in \mathbb{R}\setminus\mathbb{Q}$, there exist unique integer $r\ge 1$ and $g_1,\cdots,g_r\in\{\tau,\tau^{-1},\sigma,U\}$ such that $g= g_r\cdots g_1$ with $g_ig_{i-1}\not =1$, and
\begin{equation}\label{eq:dec.e}\begin{matrix*}[l]
x_{i-1}\in(-\infty,-1)\cup (0,1) & \text{if } g_i=\sigma, \text{ and } \vspace{0.4ex} \\
x_{i-1}\in (-1,1)                      & \text{if } g_i=U,
\end{matrix*}\end{equation}
where $x_i=g_ig_{i-1}\cdots g_1 x_{0}$.
\end{prop}
\begin{proof}
We have $U = \tau^{\pm 1} U \tau^{\pm 1}$ and $\sigma U = U \sigma$. Thus,
\begin{enumerate}
\item If $g=U$ and $x_0\in (2n-1,2n+1)$ for $n\in\mathbb{Z}$, then we have $r=2n+1$ and
$g_r\cdots g_1 = \tau^{-n}U\tau^{-n}.$
\item For $g=\sigma$ and $x_0\in (1,\infty)$, if $x_0\in (2n-1,2n+1)$ for $n\in \mathbb{N}$, we have $r=2n+3$ and
$g_r\cdots g_1 = U\sigma \tau^{-n}U\tau^{-n}.$
\item For $g=\sigma$ and $x_0\in (-1,0)$, if $x_0 \in \left(-\frac{1}{2n-1},-\frac{1}{2n+1}\right)$ for $n\in \mathbb{N}$, we have $r=2n+3$ and $g_r\cdots g_1 = \tau^n U\tau^n\sigma U.$
\end{enumerate}

To show the uniqueness, let us assume that there are $g_i$, $i=1,\dots,r$, such that $g_r\cdots g_1=I$ and $r$ is minimal.
Let $i_1<i_2<\cdots <i_k$ be indices $i$ such that $g_i=\sigma$.
For given $x_0\in \mathbb{R}\setminus \mathbb{Q}$, let us denote by $x_i = g_i x_{i-1}$.
Note that $\prod_{i=1}^r g_i'(x_{i-1}) = g'(x_0) = 1$.
Since $g_j'(x) =1$ if $g_j = \tau^{\pm1}$ and $g_{i_\ell}'(x) = x^{-2}$, we have $\prod_{\ell=1}^k |x_{i_\ell-1}| = 1$.

If $x_{i_\ell-1}<-1$, then $x_{i_\ell} = -\frac{1}{x_{i_\ell-1}}\in (0,1)$.
Then, $g_{i_\ell} = \tau$, $\tau^{-1}$ or $U$.
If $g_{i_\ell}=\tau$, then $g_{i_\ell+j}=\tau$ for all $j\ge 1$, thus $\ell=k$.
If $g_{i_\ell}=U$, then $x_{i_\ell}=\frac{1}{x_{i_\ell-1}}\in (-1,0)$.
Then $g_{i_\ell+j} = \tau^{\pm}$ for $j\le i_{\ell+1}-i_\ell-1$ and $|x_{i_{\ell+1}-1}|>1$.
If $g_{i_\ell}=\tau^{-1}$, then $g_{i_{\ell}+j}=\tau^{-1}$ for $j\le i_{\ell+1}-i_\ell-1$ and $x_{i_{\ell+1}-1}<-1$.
Since $g_{i-1}\cdots g_1 g_r\cdots g_i=I$ for all $i$, if $|x_{i_\ell-1}|>1$ for some $\ell$, then $|x_{i_\ell-1}|>1$ for all $\ell$.
It contradicts to $\prod_{\ell=1}^k|x_{i_\ell-1}|=1$.

We have $|x_{i_\ell-1}|\le 1$ for all $\ell$. If $|x_{i_\ell-1}|<1$ for some $\ell$, then it contradicts to $\prod_{\ell=1}^k |x_{i_\ell-1}| = 1.$
Therefore $x_{i_\ell-1}=1$ for all $\ell$. It contradicts to $x_0$ is irrational.
Thus, there are no $g_i$'s such that $g_r\cdots g_1=I$.
\end{proof}

\begin{coro}\label{co:auto}
Let $A$ be an abelian ring.
For three maps $t:\mathbb{R}\setminus\mathbb{Q}\to A^*,~ s:(0,1)\setminus\mathbb{Q}\to A^* \text{ and } u: (0,1)\setminus\mathbb{Q} \to A^*,$ there exists a unique automorphic factor $\chi$ such that
\begin{equation*}\begin{matrix*}[l]
\chi(\tau,x) = t(x)       & \text{for }x\in  \mathbb{R}\setminus \mathbb{Q}, \vspace{0.4ex} \\
\chi(\sigma,x) = s(x)  & \text{for }x\in (0,1)\setminus\mathbb{Q},             \vspace{0.4ex} \\
\chi(U,x) = u(x)         & \text{for }x\in (0,1)\setminus \mathbb{Q}.
\end{matrix*}\end{equation*}
\end{coro}

\begin{proof}
We define by
\begin{equation*}\label{eq:auto.e} \begin{matrix*}[l]
\chi(\tau^{-1},x) = (t(x-2))^{-1}          & \text{for }x\in \mathbb{R}\setminus\mathbb{Q},                 \vspace{0.4ex} \\
\chi(\sigma,x) = (\sigma(-1/x))^{-1}  & \text{for } x\in (-\infty,-1)\setminus\mathbb{Q}, \text{ and } \vspace{0.4ex} \\
\chi(U,x) = (u(-x))^{-1}                      & \text{for }x\in (-1,0)\setminus \mathbb{Q}.
\end{matrix*}\end{equation*}
By the same arguments as in the proof of Corollary~\ref{co:auto.b}, we have the existence and the uniqueness of such an automorphic factor.
\end{proof}

\begin{coro}
Let $A$ be an abeilan ring, $\chi$ be an automorphic factor, $M$ be an $A$-module such that $M^X$ having the $\mathbb Z^{[\widetilde\Theta]}$-module structure associated with $\chi$.
For three maps 
$\check{c}_\tau : \mathbb{R}\setminus\mathbb{Q}\to M,$ $\check{c}_\sigma: (0,1)\setminus\mathbb{Q} \to M$ and 
$\check{c}_U:(0,1)\setminus\mathbb{Q}\to M,$
there exists a unique cocycle $\check{c}: \widetilde\Theta \times (\mathbb{R}\setminus\mathbb{Q}) \to M$ such that 
\begin{equation*}\begin{matrix*}[l]
\check{c}(\tau;x) = \check c_\tau(x)          & \text{for } x\in \mathbb{R}\setminus\mathbb{Q}, \vspace{0.4ex} \\
\check{c}(\sigma;x) = \check c_\sigma(x) & \text{for } x\in (0,1)\setminus\mathbb{Q},           \vspace{0.4ex} \\
\check{c}(U;x) = \check c_U(x)                 & \text{for }x\in (0,1)\setminus \mathbb{Q}.
\end{matrix*}\end{equation*}
\end{coro}

\begin{proof}
We define by
\begin{equation*}\begin{matrix*}[l]
\check{c}(\tau^{-1};x) = -\chi(\tau^{-1},x)\check{c}_{\tau}(x-2) & \text{for }x\in\mathbb{R}\setminus\mathbb{Q}, \vspace{0.4ex} \\
\check{c}(\sigma;x) = -\chi(\sigma;x) \check{c}_{\sigma}(-1/x)   & \text{for }x\in (-\infty,-1)\setminus \mathbb{Q}, \vspace{0.4ex} \\
\check{c}(U;x) = -\chi(U,x) \check{c}_{U}(-x)           & \text{for }x\in (-1,0)\setminus\mathbb{Q}.
\end{matrix*}\end{equation*}
With the same arguments of the proof of Corollary~\ref{co:coc.b}, we have the existence and the uniqueness of such a cocycle.
\end{proof}

Let $A = \mathbb{R}$, $t(x) = 1$, $s(x) = x^\nu$ on $(0,1)$ for $\nu \in \mathbb{R}$, $u(x) = 1$ on $(0,1)$, $\check{c}_\tau(x)=0$, $\check{c}_\sigma(x) = -\log x $ on $(0,1)$ and $\check{c}_U(x) = 0$ on $(0,1)$.
There exist a unique automorphic factor $\chi$ and a unique cocycle $\check{c}$ such that
\begin{align*}
\begin{cases}
\chi(\tau^n,x) = 1 & \text{for } n\in \mathbb{Z},~ x\in \mathbb{R}\setminus\mathbb{Q}, \\
\chi(\sigma, x) = x^\nu & \text{for } x\in (0,1)\setminus\mathbb{Q}, \\
\chi(U,x) = 1 & \text{for }x\in (0,1)\setminus\mathbb{Q}, 
\end{cases}
\text{ and }
\begin{cases}
\check{c}(\tau;x) = 0 & \text{for } x\in\mathbb{R}\setminus\mathbb{Q},\\
\check{c}(\sigma;x) = -\log x & \text{for }x\in (0,1)\setminus\mathbb{Q},\\
\check{c}(U;x) = 0 & \text{for } x\in (0,1)\setminus \mathbb{Q}.
\end{cases}
\end{align*}
By \eqref{eq:even2Z}, \eqref{eq:fun.e} and \eqref{eq:cob}, the $1$-coboundary of $B_{even,\nu}$ satisfies
\begin{align*}\begin{cases}
(d^0 B_{even,\nu})(\sigma;x) = \chi(\sigma,x) B_{even,\nu}(\sigma x) - B_{even,\nu}(x) = \log x & \text{for }x\in (0,1)\setminus\mathbb Q,\\
(d^0 B_{even,\nu})(U;x) = \chi(U,x) B_{even,\nu}(U x) - B_{even,\nu}(x) = 0 & \text{for }x\in (0,1)\setminus\mathbb Q,\\
(d^0 B_{even,\nu})(\tau;x) = \chi(\tau^{-1},x)B_{even,\nu}(\tau^{-1}x)-B_{even,\nu}(x) = 0 & \text{for }x\in\mathbb R\setminus \mathbb Q.
\end{cases}
\end{align*}
Thus $\check{c} = d^0(-B_{even,\nu})$, i.e. the $1$-coboundary of $-B_{even,\nu}$ is a $1$-cocycle under the $\widetilde\Theta$-action.

\end{document}